\documentclass{amsart}

\usepackage{amsmath, amsthm}
\usepackage{amssymb}

\usepackage{cite}

\usepackage{tikz}
\usepackage{wrapfig}

\usetikzlibrary{positioning}

\theoremstyle{definition}

\newtheorem{thm}{Theorem}[section]
\newtheorem*{thm*}{Theorem}

\newtheorem{lem}[thm]{Lemma}
\newtheorem{defn}[thm]{Definition}
\newtheorem{claim}[thm]{Claim}
\newtheorem{prop}[thm]{Proposition}
\newtheorem{cor}[thm]{Corollary}

\newtheorem{remark}[thm]{Remark}
\newtheorem{fact}[thm]{Fact}

\newtheorem{obs}[thm]{Observation}
\newtheorem{conj}[thm]{Conjecture}

\renewcommand{\subset}{\subseteq}

\newcommand{\rest}{\restriction}

\newcommand\lto{\longrightarrow}
\newcommand\force{\Vdash}
\newcommand\R{\mathbb{R}}
\newcommand\N{\mathbb{N}}

\DeclareMathOperator{\dom}{dom}

\newcommand{\set}[2]{ \left\{ #1 :\, #2 \right\} }
\newcommand{\seqq}[2]{ \left\langle #1 :\, #2\right\rangle }

\title{Borel reducibility and symmetric models}
\author{Assaf Shani}
\date{\today}

\begin{document}

\maketitle

\begin{abstract}
We develop a correspondence between Borel equivalence relations induced by closed subgroups of $S_\infty$ and weak choice principles, and apply it to prove a conjecture of Hjorth-Kechris-Louveau (1998).

For example, we show that the equivalence relation $\cong^\ast_{\omega+1,0}$ is strictly below $\cong^\ast_{\omega+1,<\omega}$ in Borel reducibility.
By results of Hjorth-Kechris-Louveau, $\cong^\ast_{\omega+1,<\omega}$ provides invariants for $\Sigma^0_{\omega+1}$ equivalence relations induced by actions of $S_\infty$, while $\cong^\ast_{\omega+1,0}$ provides invariants for $\Sigma^0_{\omega+1}$ equivalence relations induced by actions of \textit{abelian} closed subgroups of $S_\infty$.

We further apply these techniques to study the Friedman-Stanley jumps.
For example, we find an equivalence relation $F$, Borel bireducible with $=^{++}$, so that $F\rest C$ is not Borel reducible to $=^{+}$ for any non-meager set $C$. This answers a question of Zapletal, arising from the results of Kanovei-Sabok-Zapletal (2013).

For these proofs we analyze the symmetric models $M_n$, $n<\omega$, developed by Monro (1973), and extend the construction past $\omega$, through all countable ordinals.
This answers a question of Karagila (2019).
\end{abstract}

\section{Introduction}
\subsection{Background}
The notion of {\bf Borel reducibility} gives a precise way of measuring the complexity of various equivalence relations.
Given equivalence relations $E$ and $F$ on Polish spaces $X$ and $Y$ respectively, we say that a map $f\colon X\lto Y$ is a \textbf{reduction} of $E$ to $F$ if for any $x,y$ in $X$, $x\mathrel{E}y\iff f(x)\mathrel{F}f(y)$.
That is, $f$ reduces the problem of determining $E$-relation to that of $F$-relation.
Given that $f$ is \textit{definable} in a simple way, we think of $E$ as less complicated than $F$.
The common definability requirement is that $f$ is a Borel map.
Say that $E\leq_B F$ ($E$ is \textbf{Borel reducible} to $F$) if there exists some Borel reduction of $E$ to $F$.

We say that $E$ and $F$ are \textbf{Borel bireducible} (in symbols $E\sim_B F$) if $E\leq_B F$ and $F\leq_B E$.
Furthermore, $E<_BF$ means that $E\leq_B F$ and $F\not\leq_B E$, and we say that $E$ is strictly below $F$ in Borel reducibility.

Another point of view comes from the notion of classification.
A \textbf{complete classification} of an equivalence relation $E$ on $X$ is a map $c\colon X\lto I$ such that for any $x,y\in X$, $x\mathrel{E}y\iff c(x)=c(y)$, where $I$ is some set of \textbf{complete invariants}.
To be useful, such map $c$ needs to be definable in a reasonable way.
If $E$ is Borel reducible to $F$ then any set of complete invariants for $F$ can be used as a set of complete invariants for $E$, thus  the invariants required to classify $E$ are no more complicated than those required to classify $F$.

A {\bf Borel equivalence relation} $E$ on a Polish space $X$ is an equivalence relation on $X$ which is Borel as a subset of $X\times X$.
An equivalence relation is {\bf classifiable by countable structures} if it is Borel reducible to an orbit equivalence relation induced by a continuous action of a closed subgroup of $S_\infty$.
This is a wide notion of being classifiable by ``reasonably concrete'' invariants.
The {\bf Friedman-Stanley jump hierarchy}, defined below, is cofinal among the Borel equivalence relations which are classifiable by countable structures, and is used to calibrate those. 

Recall that $\cong_2$, often called $=^+$ (the first Friedman-Stanley jump), is the equivalence relation on $\R^\N$ relating two sequences $\left\langle x_0,x_1,x_2,... \right\rangle$ and $\left\langle y_0,y_1,y_2,... \right\rangle$ if for any $n$ there is some $m$ such that $x_n=y_m$, and vice versa.
That is, if the two sequences enumerate the same countable set of reals.
The map sending a sequence $\left\langle x_0,x_1,x_2,... \right\rangle$ to the unordered set $\set{x_n}{n=0,1,2,...}$ is a complete classification of $\cong_2$.
The complete invariants here are all countable sets of reals.

For a countable ordinal $\alpha$, the equivalence relation $\cong_\alpha$ is defined in a similar way so that it can be classified by invariants which are the hereditarily countable sets in $\mathcal{P}^\alpha(\N)$ (see \cite{FS89}, \cite{HKL98}).
Here $\mathcal{P}^\alpha(\N)$ is the $\alpha$-iterated powerset of the natural numbers.
For example, $\cong_3$, also called $=^{++}$ (the second Friedman-Stanley jump), is defined on $\R^{\N^2}$ so that the map
\begin{equation*}
        \seqq{x_{i,j}}{i,j\in\N}\in\R^{\N^2}\mapsto \set{\set{x_{i,j}}{j\in\N}}{i\in\N}\in\mathcal{P}^3(\N)
\end{equation*}
is a complete classification.


A Borel equivalence relation $E$ is said to be of {\bf potential complexity} $\Gamma$ ($E$ is $\mathrm{pot}(\Gamma)$), for a point class $\Gamma$ (closed under continuous preimages) if there is an equivalence relation $F$ in $\Gamma$ such that $E\leq_B F$.
We say that $\Gamma$ is {\bf the potential complexity of $E$} if it is minimal such that $E$ is $\mathrm{pot}(\Gamma)$ (see \cite{HKL98} p. 65).

In \cite{HKL98}, Hjorth, Kechris and Louveau have completely classified the possible potential complexities of Borel equivalence relations classifiable by countable structures, and found them to be precisely the point classes $\mathbf{\Delta}^0_1$, $\mathbf{\Pi}^0_1$ $\mathbf{\Sigma}^0_2$, $\mathbf{\Pi}^0_n$, $D(\mathbf{\Pi}^0_n)$ ($n\geq 3$), $\bigoplus_{\alpha<\lambda}\mathbf{\Pi}^0_\alpha$, $\mathbf{\Sigma}^0_{\lambda+1}$, $\mathbf{\Pi}^0_{\lambda+n}$, $D(\mathbf{\Pi}^0_{\lambda+n})$ ($\lambda$ limit and $n\geq 2$).
Recall that $D(\Gamma)$ is the class of all sets of the form $A\setminus B$ for $A,B\in \Gamma$.

Furthermore, for each potential class which appears in the list above they find a maximal equivalence relation, among those classifiable by countable structures \cite[Corollary 6.4]{HKL98}.
For example, $\cong_n$ is maximal $\mathbf{\Pi}^0_{n+1}$  for $n\geq 2$ and $\cong_{\lambda+n}$ is maximal $\mathbf{\Pi}^0_{\lambda+n+1}$ for limit $\lambda$ and $n\geq 1$.
For the classes $D(\mathbf{\Pi}^0_n)$, $n\geq 3$, $D(\mathbf{\Pi}^0_{\lambda+n})$ and $\mathbf{\Sigma}^0_{\lambda+1}$, $\lambda$ limit and $n\geq 2$, Hjorth-Kechris-Louveau refined the Friedman-Stanley hierarchy as follows.

\subsection{The equivalence relations of Hjorth, Kechris and Louveau}\label{sec;HKL-relations}

For $n\geq 3$ and $0\leq k\leq n-2$, the equivalence relation $\cong^\ast_{n,k}$ is defined to have complete invariants in the collection $\mathcal{P}_\ast^{n,k}(\mathbb{N})$, which is defined as follows (pages 95, 98, 99 in \cite{HKL98})\footnote{The iterated powersets $\mathcal{P}^n(\mathbb{N})$ are defined slightly differently in \cite{HKL98}, by taking $\mathcal{P}^{k+1}(\mathbb{N})=\mathcal{P}(\mathcal{P}^k(\mathbb{N})\cup\mathbb{N})$. This does not present a serious difference, see Remark~\ref{remark;cong-ast-defn}}.
The members of $\mathcal{P}_\ast^{n,k}(\mathbb{N})$ are pairs $(A,R)$ such that:
\begin{itemize}
    \item $A$ is a hereditarily countable set in $\mathcal{P}^n(\N)$ (i.e., a $\cong_n$-invariant);
    \item $R$ is a ternary relation on $A\times A\times(\mathcal{P}^k(\N)\cap \mathrm{tc}(A))$;
    \item  given any $a\in A$, for any $b,b'\in A$ and any $x\in\mathcal{P}^k(\mathbb{N})$, if $R(a,b,x)$ and $R(a,b',x)$ holds then $b=b'$. Furthermore for any $a,b\in A$ there is some $x$ such that $R(a,b,x)$ holds.
\end{itemize}
A useful intuition is as follows: if one assumes additionally that for any $a,b\in A$ there is a unique $x$ such that $R(a,b,x)$ holds, then the third condition says that $R(a,-,-)$ is an injective map. Thus $R$ allows to code $A$ into a lower rank set, uniformly in a parameter $a\in A$.
This can always be achieved if $\beta=0$. Furthermore, this will be true for most pairs $(A,R)$ we consider in this paper.
In general, we get injective maps $b\mapsto \set{x}{R(a,b,x)}$, from $A$ to $\mathcal{P}^{k+1}(\mathbb{N})$ such that the sets in the image are disjoint. (See the remarks after Claim~\ref{claim;cong-ast-invs} for how the disjointness is used.)

The idea is that the members of $\mathcal{P}_\ast^{n,k}(\mathbb{N})$ have complexity intermediate between $\mathcal{P}^n(\mathbb{N})$ and $\mathcal{P}^{n-1}(\mathbb{N})$.
A good intuition comes from countable equivalence relations: if $E$ is a countable Borel equivalence relation, then the map $x\mapsto[x]_E=A_x$ is a classification by countable structures, using invariants in $\mathcal{P}^2(\mathbb{N})$.
These invariants have additional structure: using a parameter $a\in A_x$, the set $A_x$ can be definably enumerated, and therefore coded as a single real, that is, a member of $\mathcal{P}^1(\mathbb{N})$.

The Borel equivalence relations $\cong^\ast_{n,k}$ are defined precisely so that they admit a natural complete classification with complete invariants in $\mathcal{P}_\ast^{n,k}(\mathbb{N})$. (See \cite{HKL98} p.98-99 and p.95 for an explicit Polish space $P_\ast^{n,k}$ on which $\cong^\ast_{n,k}$ is defined. 
More specifically, it is the space of countable structures coding a set in $\mathcal{P}^n(\mathbb{N})$ together with a relation satisfying the three conditions above.)
The treatment of $\cong^\ast_{n,k}$ here is in terms of invariants, and the particular coding in a Polish space does not play a role below.


For a limit ordinal $\lambda$, complete invariants for the equivalence relations $\cong^\ast_{\lambda+1,<\lambda}$, $\cong^\ast_{\lambda+1,\beta}$ for $\beta<\lambda$ and 
$\cong^\ast_{\lambda+n,\beta}$ for $n\geq 2$, $\beta\leq\lambda+n-2$ are defined in a similar way, such that for $\beta\leq\gamma$, $\cong^\ast_{\alpha,\beta}\leq_B\cong^\ast_{\alpha,\gamma}\leq_B\cong_\alpha$. (See \cite{HKL98} p.99, also p.67.)

Hjorth-Kechris-Louveau show that, among Borel equivalence relations classifiable by countable structures, $\cong^\ast_{n,n-2}$ is maximal $D(\mathbf{\Pi}^0_n)$ for $n\geq 3$, $\cong^\ast_{\lambda+1,<\lambda}$ is maximal $\mathbf{\Sigma}^0_{\lambda+1}$ and $\cong^\ast_{\lambda+n,\lambda+n-2}$ is maximal $D(\mathbf{\Pi}^0_{\lambda+n})$ for limit $\lambda$ and $n\geq 2$.
Moreover, they show that a $D(\mathbf{\Pi}^0_n)$ (respectively $\mathbf{\Sigma}^0_{\lambda+1}$, $D(\mathbf{\Pi}^0_{\lambda+n})$) equivalence relation which is induced by an action of an \textit{abelian} closed subgroup of $S_\infty$ is in fact Borel reducible to $\cong^\ast_{n,0}$ (respectively $\cong^\ast_{\lambda+1,0}$, $\cong^\ast_{\lambda+n,0}$).
They conjecture that all these hierarchies are strict with respect to Borel reducibility.

\begin{conj}[Hjorth-Kechris-Louveau \cite{HKL98}, p. 104]\label{conj;HKL}
\hfill
\begin{enumerate}
    \item For any $n\geq 3$, $l<k\leq n-2$, $\cong^\ast_{n,l}<_B\cong^\ast_{n,k}$;
    \item For limit $\lambda$, $\alpha<\beta<\lambda$, $\cong^\ast_{\lambda+1,\alpha}<_B\cong^\ast_{\lambda+1,\beta}$;
    \item For $\lambda$ limit, $n\geq 2$, $\alpha<\beta\leq\lambda+n-2$, $\cong^\ast_{\lambda+n,\alpha}<_B\cong^\ast_{\lambda+n,\beta}$.
\end{enumerate}

\end{conj}

Hjorth-Kechris-Louveau prove that for any countable ordinal $\alpha$, $\cong^\ast_{\alpha+3,\alpha}<_B\cong^\ast_{\alpha+3,\alpha+1}$ \cite[Theorem 6.6]{HKL98}.
All other instances of the conjecture, for example, whether $\cong^\ast_{4,0}$ and $\cong^\ast_{4,1}$ are distinct, remained open.
The most important instances of the conjecture are for $\omega+1$ and $\omega+2$, in which case the results in \cite{HKL98} left open whether $\cong^\ast_{\omega+1,0}$ is different than $\cong^\ast_{\omega+1,<\omega}$ and whether $\cong^\ast_{\omega+2,0}$ is different than $\cong^\ast_{\omega+2,\omega}$.
That is, whether invariants for $\mathbf{\Sigma}^0_{\omega+1}$ (respectively $D(\mathbf{\Pi}^0_{\omega+2})$) equivalence relations induced by \textit{abelian} group actions are genuinely simpler (see \cite[p.68]{{HKL98}}). 

The central result of this paper is to verify the conjecture above.

\begin{thm}\label{thm;HKL-conj}
Conjecture~\ref{conj;HKL} is true.
\end{thm}

Part (1) of the conjecture is established in Section~\ref{subsec;cong-ast-n-k}, Corollary~\ref{cor;conj-part-1}.
For parts (2) and (3), we focus on showing $\cong^\ast_{\omega+1,n}<_B\cong^\ast_{\omega+1,<\omega}$ for any $n<\omega$ and that $\cong^\ast_{\omega+2,<\omega}<_B\cong^\ast_{\omega+2,\omega}$.
These two results are dealt with in sections \ref{subsec;cong-ast-omega+1} and \ref{subsec;cong-ast-omega+2} respectively.
The general proof of parts (2) and (3) is outlined in Section~\ref{sec;general-conj-proof}.

The techniques developed for proving the above mentioned irreducibility results are flexible.
In particular, they can be adapted to show that

\begin{itemize}
    \item $\cong^\ast_{\alpha+1,\beta+1}\not\leq_B (\cong^\ast_{\alpha+1,\beta})^\omega$;
    \item $(\cong^\ast_{\alpha+1,0})^\omega\not\leq_B \cong^\ast_{\alpha+1,\beta}$;
    \item $(\cong^\ast_{\alpha+1,\beta})^\omega<_B\cong_{\alpha+1}$;
\end{itemize}
where $2\leq\alpha$, $1\leq\beta$ and $\beta+1<\alpha$.
These properties of $\cong^\ast_{\alpha+1,\beta}$, with respect to $\cong_{\alpha+1}$, are similar to the behavior of countable Borel equivalence relations with respect to $\cong_2$.

\subsection{Generic reductions and homomorphisms}\label{subsec;gen-red}


The following definition captures those equivalence relations whose complexity is based on Baire category arguments.

\begin{defn}[Kanovei-Sabok-Zapletal \cite{ksz}, Definition 1.16\footnote{Kanovei-Sabok-Zapletal study the behaviour of equivalence relations on $I$-positive sets for various ideals $I$. We only mention the case where $I$ is the meager ideal here.}]
An analytic equivalence relation $E$ is {\bf in the spectrum of the meager ideal} if there is an equivalence relation $F$ on a Polish space $Y$ which is Borel bireducible with $E$, and furthermore for any non meager set $C\subset Y$, $F\rest C$ is Borel bireducible with $E$.
\end{defn}
For example, the equivalence relations which admit a dichotomy theorem, such as $E_0$, $E_1$ and $E_0^\N$, are in the spectrum of the meager ideal, as witnessed by the standard product topology on $2^\N$, $\R^\N$ and $(2^\N)^\N$ respectively.
Moreover, this fact is crucial in the proof of the dichotomy theorems (see \cite{HKL90}, \cite{KL97} and \cite{HK97}).
Kanovei-Sabok-Zapletal \cite[Theorem 6.23]{ksz} show that $=^+$ is in the spectrum of the meager ideal, as witnessed by the standard product topology on $\R^\N$.

Very few natural equivalence relations are known to be in the spectrum of the meager ideal (see page 6 of \cite{ksz}).
In particular, Zapletal asks (private communication) whether $=^{++}$ is in the spectrum of the meager ideal, which was left open by the results of \cite{ksz}.
In fact, it was not known whether the irreducibility proof $=^{++}\not\leq_B=^+$ holds on comeager sets.
That is, the known proofs of this irreducibility, \cite{FS89}, \cite{HKL98} and \cite{zprep}, do not involve Baire category arguments.


\begin{prop}[Proposition~\ref{prop;reduction-of-++} below]
Consider $\R^{\N^2}$ equipped with the standard product topology.
There is a comeager $C\subset \R^{\N^2}$ such that $=^{++}\rest C$ \textbf{is} Borel reducible to $=^+$.
\end{prop}


On the other hand, we find a different presentation of $=^{++}$, susceptible to Baire-category arguments.

\begin{prop}[Proposition~\ref{prop;=++-no-reduction} below]\label{prop;=++-no-reduction-intro}
There is an equivalence relation $F$, Borel bireducible with $=^{++}$, such that $F\rest C\not\leq_B =^+$ for any nonmeager set $C$.
\end{prop}

The equivalence relation $F$ (defined in Section~\ref{sec;topologies}) is based on a new proof of the irreducibility $=^{++}\not\leq_B =^+$, presented in Section~\ref{sec;Monro-models}.
It can be shown further that $F$ retains its complexity on non-meager sets, thus $=^{++}$ is in the spectrum of the meager ideal.
A proof will appear in future work.

Our methods further apply to study homomorphisms between equivalence relations.
Given equivalence relations $E$ and $F$ on $X$ and $Y$ respectively, a map $f\colon X\lto Y$ is a \textbf{homomorphism} from $E$ to $F$ if for any $x,y\in X$, if $xEy$ then $f(x)Ff(y)$.
A Borel homomorphism between $E$ and $F$ corresponds to a definable map between $E$-invariants and $F$-invariants which is not necessarily injective.

Kanovei-Sabok-Zapletal established the following strong structural result about homomorphisms of $=^+$.

\begin{thm}[Kanovei-Sabok-Zapletal, a corollary of Theorem 6.24 \cite{ksz}]\label{thm;ksz-thm-6-24}
Let $E$ be an analytic equivalence relations. Then either
\begin{itemize}
    \item $=^+$ is Borel reducible to $E$, or
    \item any Borel homomorphism from $=^+$ to $E$ maps a comeager set into a single $E$-class (in which case we say that $=^+$ is {\bf generically $E$-ergodic}).
\end{itemize}
\end{thm}

This result cannot generalize to $=^{++}$, as there is a non-trivial homomorphism from $=^{++}$ to $=^+$.
That is, the union map, which we denote by $u$, sending a set of sets of reals $A$ to its union $\bigcup A$.
With the presentation of $=^{++}$ verifying Proposition~\ref{prop;=++-no-reduction-intro}, we can show that this is the only interesting homomorphism from $=^{++}$ to $=^+$, generically. 
\begin{wrapfigure}{r}{2.5cm}
\begin{tikzpicture}[
node/.style={}]
\node[node]      (dplus)           {$F$};
\node[node]     (uplus) [below=of dplus]
{$=^+$};
\node[node]     (fplus) [right=of uplus]
{$=^+$};

\draw[->]  (dplus.south) --node[anchor=west] {$u$} (uplus.north) ;
\draw[->] (dplus.south east) -- node[anchor=south west] {$f$} (fplus.north west)   ;
\draw[->,dashed]    (uplus.east) --node[anchor=south] {$h$} (fplus.west) ;
\end{tikzpicture}
\end{wrapfigure}

\begin{thm}[Proposition~\ref{prop;=++-unique-homo} below]\label{thm;=++-unique-homo}
There is an equivalence relation $F$, Borel bireducible with $=^{++}$ such that for any homomorphism $f$ from $F$ to $=^+$ there is a homomorphism $h$ from $=^+$ to $=^+$ such that, on a comeager set, $f(x)=^{+} h\circ u(x)$.

\end{thm}

Similar results for the higher Friedman-Stanley jumps also follow from the results in this paper.

\subsection{}

A brief outline of our approach is as follows.
In this paper we develop a correspondence between the study of Borel equivalence relations, up to Borel reducibility, and the study of symmetric models and fragments of the axiom of choice.
Thus providing new tools for proving Borel irreducibility results (sections \ref{subsec;connection}, \ref{sec;double-brackets}).

With this translation we will see that the models developed by Monro in 1973 are closely related to the finite Friedman-Stanley jumps $\cong_n$ (sections \ref{subsec;Kinna-Wagner}, \ref{sec;Monro-models}) and use them to conclude the results about $=^{++}$ (section~\ref{sec;topologies}).
We further show how to study the relations $\cong^\ast_{n,k}$ with these methods and establish part (1) of Conjecture~\ref{conj;HKL} (section \ref{sec;cong-ast-n-k}). 

It will then be evident that a proof of Conjecture~\ref{conj;HKL} relies on extending Monro's construction past $\omega$.
The latter problem was recently asked by Karagila and is closely related to some recent developments in symmetric models (see section \ref{subsec;Kinna-Wagner}).
A considerable chunk of this paper is then devoted to solve this problem and conclude Theorem~\ref{thm;HKL-conj} (sections \ref{sec;trans-jumps}, \ref{sec;general-conj-proof}).

\subsection{A connection with symmetric models}\label{subsec;connection}

To illustrate the relationship with symmetric models we recall first the ``basic Cohen model'' in which the axiom of choice fails.
Let $x=\left<x_0,x_1,...\right>$ be a sequence of Cohen reals generic over some base model $V$ and let $A=\set{x_n}{n=0,1,...}$ be the unordered collection of these reals.
The basic Cohen model can be seen as the closure of $A$ over $V$ under definable set-theoretic operations, denoted $V(A)$.
Based on earlier results of Fraenkel and Mostowski, Cohen \cite{Coh63} shows that the set $A$ cannot be well ordered in $V(A)$, hence the axiom of choice fails.

Note that the set $A$ is simply the $=^+$-invariant of the generic sequence $x\in\R^\N$.
Let $=_\R$ be the equality relation on $\R$.
It is well known that given a real $r$ (an $=_\R$-invariant) in some generic extension of $V$, the model $V(r)$ does satisfy the axiom of choice. 
This draws a distinction between $=^+$ and $=_\R$.


More generally, the idea is as follows.
Let $E$ be an equivalence relation on a Polish space $X$, and $x\mapsto A_x$ a complete classification of $E$, witnessing that $E$ is classifiable by countable structures.
Given a generic real $x\in X$, it corresponds to a {\bf generic $E$-invariant} $A=A_x$.
We will study the set-theoretic definable closure of this $E$-invariant, $V(A)$.

Suppose $F$ is another equivalence relation, classifiable by countable structures, and there is a Borel reduction of $E$ to $F$.
This corresponds to a definable injective map between $E$-invariants and $F$-invariants.
In particular, the $E$-invariant $A$ is mapped to some $F$-invariant $B$ which is definable from $A$, and $A$ can be definably recovered from $B$.

We conclude that $B$ is in the definable closure of $A$, $V(A)$, and furthermore it generates the whole model: $V(A)=V(B)$.
We stress the contrapositive, which is the central tool used
in this paper (see Lemma~\ref{lem;reduction-def-generating-set} for a more precise statement):

\begin{thm}\label{thm;intro-corresp}
To show that there is no Borel reduction of $E$ to $F$, it suffices to find some generic $E$-invariant $A$ so that the model $V(A)$ is not generated by any $F$-invariant.
\end{thm}
The converse is also true, generically. (See Theorem~\ref{thm;correspondence} for a more precise statement.)

\begin{thm}\label{thm;intro-corresp-equiv}
The following are equivalent:
\begin{itemize}
    \item There is a partial reduction of $E$ to $F$ defined on a non-meager set;
    \item There is a Cohen-generic $x\in X$ and some $F$-invariant $B$ in $V(A_x)$ such that $V(A_x)=V(B)$ and $A_x$ and $B$ are bi-definable. 
\end{itemize}
\end{thm}

\begin{remark}
The study of definable invariants and their complexity is fundamental to the theory of Borel equivalence relations.
The notion of Borel reducibility seems to capture this basic intuition (see \cite{FS89}, \cite{Kec92}, \cite{HK96}, \cite{HKL98}).
Theorems \ref{thm;intro-corresp} and \ref{thm;intro-corresp-equiv} can be seen as further justification of this idea.
Here the notion of definability is quite generous: we only ask that given a single $E$-invariant $A$, it defines in some set-theoretic way an $F$-invariant which codes $A$.
        
\end{remark}

The notion of symmetric models was first introduced by Fraenkel~\cite{Fra22} to argue for the independence of the axiom of choice and was further developed by Fraenkel \cite{Fra37}, Mostowski \cite{Mos39}, and others to study the relationship between fragments of the axiom of choice.
Cohen's method of forcing together with the earlier Fraenkel-Mostowski techniques initiated a large industry of independence results among fragments of the axiom of choice, still active today.
An encyclopedic summary of these results can be found in \cite{HR98} (see also \cite{Jec73}, \cite{Kan06}, \cite{Kan08}).

The main point we hope to convey is that the correspondence between the two fields, together with the almost hundred years of developments in symmetric models, provides a powerful tool to the study of Borel equivalence relations.
We believe that the methods developed in this paper will find further applications and contribute both to the study of Borel equivalence relations and the study of symmetric models.

\subsection{Kinna-Wagner principles}\label{subsec;Kinna-Wagner}

In this section we introduce the generalized Kinna-Wagner principles.
We review their original motivation and their modern relevance in connection with Woodin's Axiom of Choice Conjecture.

\begin{defn}[Monro \cite{Mon73}]
The \textbf{n'th Kinna-Wagner Principle} ($\mathbf{KWP}^n$) is the statement
\begin{equation*}
    \textrm{For any set }X\textrm{ there is an ordinal } \eta \textrm{ and a 1-to-1 function }f\colon X\lto \mathcal{P}^{n}(\eta).
\end{equation*}
\end{defn}
For example, $\mathrm{KWP}^0$ states that any set can be embedded in an ordinal, which is equivalent to the axiom of choice.
This definition is motivated by the following.

\textbf{Selection principle}\footnote{Often referred to as the Kinna-Wagner selection principle in the literature. According to \cite{Jec73} it was originally considered by Kuratowski.}: Suppose $X$ is a set all of whose elements are sets of size at least 2. Then there is a function $f$ defined on $X$ such that $f(x)$ is a non-empty proper subset of $x$ for any $x\in X$.

Kinna and Wagner \cite{KW55} showed that this selection principle is equivalent to $\mathrm{KWP}^1$, i.e. that any set $X$ can be embedded into the powerset of some ordinal.
Halpern and Levy \cite{HL64} show that $\mathrm{KWP}^1$ holds in the basic Cohen model.
Monro \cite{Mon73} constructed models $M_n$ in which $\mathrm{KWP}^{n-1}$ fails.

We will show that Monro's models are closely related to the finite Friedman-Stanley jumps, based on the results from Section~\ref{sec;double-brackets}.
That is, $M_n$ is of the form $V(A^n)$ where $A^n$ is a generic $\cong_{n+1}$-invariant, and is not generated by a simpler invariant.
The main difficulty will be to generalize these methods in order to study the transfinite jumps.

The principles $\mathrm{KWP}^\alpha$ are defined analogously for any ordinal $\alpha$ (see \cite{Kar16}).
A closely related notion comes from Woodin's Axiom of Choice Conjecture.
They both measure how far the model is from satisfying the axiom of choice.

\begin{defn}[See \cite{WDR16}, Definition 29]
The {\bf Axiom of Choice Conjecture} asserts in ZF, that if $\delta$ is an extendible cardinal then the axiom of choice holds in $V[G]$, where $G$ is generic over $V$ for collapsing $V_\delta$ to be countable.
\end{defn}

The axiom of choice conjecture and its relationship to Woodin's HOD conjecture were addressed in a workshop in Bristol, 2011.
In particular, the group in Bristol looked at a possible failure of the conclusion in the conjecture.
The axiom of choice conjecture implies that if $\delta$ is extendible then $\mathrm{KWP^{\delta}}$ holds.
At that time however, there was no known model in which $\mathrm{KWP}^\omega$ fails, regardless of large cardinals.
Towards that end, they sketched a construction of a model in which $\mathrm{KWP}^\alpha$ fails for all ordinals $\alpha$.
This construction was developed and analyzed by Karagila in \cite{Kar17}.

In this so called ``Bristol Model'' there are no extendible cardinals so it does not serve as a counter example to the axiom of choice conjecture.
In fact, the construction relies on certain combinatorial objects in the ground model which are incompatible with large cardinals (such as supercompacts).
Even to establish the failure of $\mathrm{KWP}^{\omega}$, some non trivial combinatorial objects are used, and the sets of high Kinna-Wagner rank do not look like invariants for Friedman-Stanley jumps.


A finer understanding of how to construct models with high failures of Kinna-Wagner principles is required to better understand the Axiom of Choice Conjecture, from this perspective.
That is, to get lower bounds for how large the cardinal $\delta$ needs to be.
We present here a different method of getting the failure of $\mathrm{KWP}^\omega$, using a significantly simpler construction than in the Bristol model and without any ground model assumptions.
Furthermore, our construction extends Monro's models and is based on Friedman-Stanley invariants.

The question of extending Monro's techniques was addressed by Karagila~\cite{Kar16}, where he casts Monro's construction as a symmetric iteration and produces a limit model $M_\omega$ satisfying $\mathrm{KWP}^\omega$ and $\neg\mathrm{KWP}^n$ for any $n<\omega$.
Karagila then asks whether Monro's construction can be continued to stage $\omega+1$ such that $\mathrm{KWP}^\omega$ fails, noting that Monro's method cannot be used directly.

\begin{thm}\label{thm;Kinna-Wagner-separation}
For any $\omega\leq\alpha<\omega_1$ there is a ``Monro-style'' model $M_\alpha=V(A^\alpha)$.
\begin{itemize}
    \item $A^\alpha$ is a generic $\cong_\alpha$-invariant;
    \item $V(A^\alpha)$ is not of the form $V(B)$ for any $\cong_\beta$-invariant $B$, $\beta<\alpha$;
    \item $V(A^{\alpha+1})\models\mathrm{KWP}^{\alpha+1}\wedge\neg\mathrm{KWP}^{\alpha}$;
    \item The construction works without any ground model assumptions.
\end{itemize}
\end{thm}


Section~\ref{sec;trans-jumps} deals with $M_{\omega+1}$.
The proof for arbitrary $\alpha$ is outlined in Section~\ref{subsec;cong-lambda}.
These models will be necessary in order to study the equivalence relations of Hjorth-Kechris-Louveau and for the proof of Theorem~\ref{thm;HKL-conj}.

\subsection*{Acknowledgements}
I would like to express my sincere gratitude to my advisor, Andrew Marks, for his guidance and support, and for numerous enlightening discussions.

I would also like to thank Yair Hayut, Menachem Magidor and Zach Norwood for many helpful conversations.

Final thanks to the referee for some helpful comments and corrections.
\section{Preliminaries}\label{sec;prelm}

We use the standard development of forcing as in \cite{Jec03}.
Our approach to symmetric models will rely on the following well known fact.

\begin{fact}[Folklore, see \cite{Mon73}, \cite{Gri75}]
Let $V$ be a ZF-model.
Suppose $A$ is a set in some generic extension of $V$.
Then there exists a {\bf minimal transitive model of ZF containing $V$ and $A$, denoted  $V(A)$.}
\end{fact}


For the results in this paper it suffices to consider models of the form $V=L(X)$ (the Hajnal relativized L construction) for some set $X$.
In this case $V(A)$ is $L(X,A)$.

We will use standard facts about {\bf ordinal definability} (see \cite{Jec03}).
Working in some generic extension $V[G]$, let $\mathrm{HOD}(V,A)$ be the collection of all sets which are hereditarily definable from $A$, parameters from the transitive closure of $A$, and parameters from $V$.
Then $\mathrm{HOD}(V,A)$ is a model of ZF, extending $V$ and containing $A$.
In the examples considered in this paper $V(A)$ and $\mathrm{HOD}(V,A)$ will coincide.

More important for the development below is that when $\mathrm{HOD}(V,A)$ is calculated inside $V(A)$ (instead of $V[G]$), the resulting model must be $V(A)$ again.
That is: 

\begin{center}
for any $X\in V(A)$, there is a formula $\psi$, parameters $\bar{a}$ from the transitive closure of $A$ and $v\in V$ such that $X$ is the unique set satisfying $\psi(X,A,\bar{a},v)$ in $V(A)$.  
\end{center}

Equivalently, for any set $X\in V(A)$ there is a formula $\varphi$, parameters $\bar{a}$ from the transitive closure of $A$ and $v\in V$ such that $X=\set{x}{\varphi(x,A,\bar{a},v)}$.

\begin{remark}\label{remark-symmetricm-presentation}
The presentation of symmetric models in this paper, as the set-theoretic definable closure of a generic object, goes back to Halpern and Levy \cite{HL64}, and is dominant in \cite{Mon73}, \cite{Gri75} and \cite{Bla81}.
It is not the standard presentation in the literature (e.g., \cite{Jec73}, \cite{Jec03}, \cite{HR98}, \cite{Kar16}, \cite{Kar17}), but is familiar.
For example, the equivalence of the two approaches for the basic Cohen model is well known (see \cite{Jec03}).
In this paper, as done above, we emphasize this approach because it is in this form that the relationship with equivalence relations and the study of invariants is most evident (see Section~\ref{sec;double-brackets}).
\end{remark}


The following ``mutual genericity'' lemma will be used.
It generalizes the fact that if $G,H$ are $P\times P$-generic over $N$, then $N[G]\cap N[H]=N$.

\begin{lem}[Folklore]\label{lem;mutgen}
Suppose $N\subset M$ are models of ZF, $P\in N$ is a poset. If $G$ is $P$-generic over $M$, then $N[G]\cap M=N$.
\end{lem}

Given a poset $Q$ and an index set $I$, let $P=\mathrm{FS}(Q,I)$ be the poset of all functions $p\colon \dom p\to Q$ where $\dom p$ is a finite subset of $I$ ($P$ is the finite support product of $I$ copies of $Q$).
For $p,q\in P$, say that $p$ extends $q$ if $\dom p\supseteq \dom q$ and $p(i)$ extends $q(i)$ in $Q$ for all $i\in \dom q$.

Let $V$ be some model of ZF, $Q,I,P$ in $V$ as above.
Given a $P$-generic filter $G$ over $V$, we may think of $G$ as a function with domain $I$ so that $G(i)$ is $Q$-generic over $V$ for each $i\in I$.
As common in many symmetric model constructions, we will be interested in the unorderd collection of these generics, $\set{G(i)}{i\in I}$.

For example, if $Q$ is Cohen forcing for adding a single real in $2^\omega$ and $I=\omega$, then $P$ is Cohen forcing for adding a generic Cohen real in $(2^\omega)^\omega$. 
Given a $P$-generic $G$, each $G(i)$ can be interpreted as a generic real and $A=\set{G(i)}{i\in\omega}$ is the unordered Cohen set of reals, that is, $V(A)$ is the basic Cohen model (assuming we start with a model $V$ of ZFC).

A well known property of the basic Cohen model $V(A)$ is that for every set $X\subset V$ there is a finite set $\bar{a}\subset A$ (often called a \textbf{support} for $X$) such that $X\in V[\bar{a}]$ (see Proposition 1.2 in \cite{Bla81}).
The standard proof works in the following more general setting:

\begin{lem}[see the arguments in \cite{Bla81}, Proposition 1.2]\label{lem;symmetry}
Let $V$ be a ZF model, $I\in V$ some index set, $Q\in V$ a poset and $P=\mathrm{FS}(Q,I)$ as above.
Let $G$ be $P$-generic over $V$ and define
\begin{equation*}
    A=\set{G(i)}{i\in I}.
\end{equation*}
Suppose $X\in V(A)$ and $X\subset V$. 
Fix $\bar{a}=a_1,...,a_n$ with $a_i\in A$, $\psi$ a formula and $v\in V$ such that $X$ is defined in $V(A)$ by $\psi(X,A,\bar{a},v)$.
(see the discussion before Remark~\ref{remark-symmetricm-presentation}). 
Then 
\begin{equation*}
   X\in V[a_1,...a_n]. 
\end{equation*}
Furthermore, $X$ is definable in $V[a_1,...,a_n]$ using $P$, $a_1,...,a_n$ and $w$.
\end{lem}
\begin{proof}[Proof sketch]
Let $i_1,...,i_n\in I$ be such that $a_k=G(i_k)$. 
Suppose $p,q\in P$ agree on $\bar{a}$, that is, $p(i_k)$ and $q(i_k)$ are compatible for each $k=1,...,n$. 
Then for any $x\in V$, $p$ and $q$ cannot force incompatible statements about whether $\check{x}\in\dot X$.
(This is because we can send $p$ to a condition compatible with $q$, by applying a finite permutation of $I$, fixing $i_1,...,i_n$ and therefore fixing $\dot{X}$.)

It follows that $X$ can be defined in $V[a_1,...,a_n]$ as the set of all $x\in V$ for which there is some condition $p\in P$ such that $p(i_k)\in a_k$ for $k=1,...,n$ and $p\force \check{x}\in\dot{X}$. 
\end{proof}

Given a Polish space $X$ we can interpret it in any generic extension.
In all examples considered here it will be very clear how to do so (e.g. $X=\R^\mathbb{N}$). See \cite{Zap08}, \cite{ksz} or \cite{Kano08}.
Similarly, given a Borel function $f$ or an analytic equivalence relation $E$, we extend their definitions in the generic extension.

If $x$ and $x'$ live in some generic extensions of $V$, we say that $x\mathrel{E}x'$ if there is a big generic extension containing both $x$ and $x'$ in which it holds.
This is well defined by the usual absoluteness arguments.

Recall that an equivalence relation is \textbf{classifiable by countable structures} if it is Borel reducible to the isomorphism relation on the space of models of a countable language (see \cite[12.3]{Kano08}, \cite{Gao09}, \cite{Hjo00}).
For the purposes of this paper, the crucial property of such equivalence relations is the following.
\begin{fact}\label{fact;CBCS-abs-classification}
Suppose $E$ is an equivalence relation on $X$ which is classifiable by countable structures. Then $E$ admits a complete classification $x\mapsto A_x$ which is definable (using some set-theoretic formula) and is absolute.
That is, if $V\subseteq M$ are models of ZF then the definition of the map $x\mapsto A_x$ is still a complete classification in $M$, and furthermore for $x\in V\cap X$ the invariant $A_x$ calculated in $M$ is the same as the one calculated in $V$.
\end{fact}

The main non-example of a classification as above is the trivial one: $x\mapsto [x]_E=\set{y\in X}{x\mathrel{E}y}$.
This is definable and remains a complete classification in any extension. However, if $E$ is not a countable equivalence relation, the orbit $[x]_E$ will change if reals are added. Thus the absoluteness required above fails.

For the equivalence relations we deal with the fact above is quite clear.
For example,
the complete classification for $\cong_2$ written in the introduction is the map $x\mapsto\set{x(n)}{n\in\mathbb{N}}$ from $\mathbb{R}^\mathbb{N}$ to countable sets of reals. The calculation of $\set{x(n)}{n\in\mathbb{N}}$ from $x\in\mathbb{R}^{\mathbb{N}}$ is absolute.
Similarly the complete classification of $\cong_3$ written in the introduction satisfies the requirements in Fact~\ref{fact;CBCS-abs-classification}.
For the higher Friedman-Stanley jumps $\cong_\alpha$, any natural classification using sets in $\mathcal{P}^\alpha(\mathbb{N})$ satisfy the fact above as well.
By \cite{BK96} (also \cite{HKL98}), if $E$ is Borel and is classifiable by countable structures, then $E$ is Borel reducible to $\cong_\alpha$ for some $\alpha<\omega_1$, so again we get a classification as in the fact above.
The equivalence relations studied in this paper are Borel and all admit similarly simple complete classifications.

Generally, for an analytic equivalence relation classifiable by countable structures, the complete classification in Fact~\ref{fact;CBCS-abs-classification} is witnessed by the Scott analysis (see \cite{Hjo00}).

For an equivalence relation $E$ we say that a complete classification $x\mapsto A_x$ \textbf{witnesses that $E$ is classifiable by countable structures} if it is definable in an absolute way as in Fact~\ref{fact;CBCS-abs-classification}.




\subsection{Notation}
We use $\omega$ to denote the set of natural numbers $\N=0,1,2,...$

For a set $A$, $A^{<\omega}$ is the set of all finite sequences of elements from $A$ and $[A]^{<\aleph_0}$ is the set of all finite (unordered) subsets of $A$.

For equivalence relations $E$ and $F$, we write $E\leq F$ instead of $E\leq_B F$, as no other notion of reduction is used in this paper.

When dealing with the finite Friedman-Stanley jumps we will often use the common notation $=^{+n}$ for $\cong_{n+1}$.

When no poset is involved, we say that a set $X$ is generic over $V$ if $X$ is in some forcing extension of $V$.

If $x$ is a real in some generic extension of $V$ then $x$ is in fact $P$-generic over $V$ for some poset $P$.
In this case we write $V[x]$ for $V(x)$.

For a formula $\psi$ and a model $M$, we denote the relativization of $\psi$ to $M$ by $\psi^M$.

\section{Symmetric models and Borel reducibility}\label{sec;double-brackets}

In this section we develop the main tools which will be used to prove Theorem~\ref{thm;HKL-conj}. 
These will allow us to translate the questions into questions about symmetric models.
The correspondence will go through the {\bf double brackets model $V[[x]]_E$} defined by Kanovie-Sabok-Zapletal~\cite{ksz}.
First we introduce the double brackets model and briefly review its original use.
We then establish the relationship with symmetric models, and develop a connection with Borel reducibility.

\begin{defn}[Kanovei-Sabok-Zapletal~\cite{ksz}, Definition 3.10]\label{def;double-brackets}
Let $E$ be an analytic equivalence relation on a Polish space $X$, and let $x\in X$ be generic over V.
Then
\begin{equation*}
    V[[x]]_E=\bigcap\set{V[y]}{y\textrm{ is in some further generic extension, }y\in X\textrm{ and }xEy}.
\end{equation*}
That is, a set $b$ is in $V[[x]]_E$ if in any generic extension of $V[x]$ and any $y$ in that extension which is $E$-equivalent to $x$, $b$ is in $V[y]$.
\end{defn}

\begin{thm}[Kanovei-Sabok-Zapletal~\cite{ksz}, Theorem 3.11]\label{thm;ksz311}
Let $E$ be an analytic equivalence relation on a Polish space $X$, and let $x\in X$ be generic over V.
\begin{itemize}
    \item $V[[x]]_E$ is a model of ZF, $V\subset V[[x]]_E\subset V[x]$.
    \item If $x,y\in X$ are generic over $V$ and $xEy$, then $V[[x]]_E=V[[y]]_E$.
\end{itemize}
\end{thm}

Kanovei-Sabok-Zapletal~\cite{ksz} study canonization properties of equivalence relations with respect to various ideals on their domain.
The central tool used there is the {\bf single brackets model $V[x]_E$} (see Definition 3.2 in~\cite{ksz}), which is a model of ZFC.
The main use of the double brackets model is in Theorem 4.22 of \cite{ksz}.
Roughly speaking, they show that for certain ideals $I$ over $X$, if $x$ is $P_I$-generic then
$V[x]_E$ and $V[[x]]_E$ agree, and then invoke Theorem 3.12 of \cite{ksz} to conclude that certain reals are in $V[x]_E$. (Recall that $P_I$ is the poset of all Borel subsets of $X$ which are not in $I$.)
Similarly, in Claim 4.19 of \cite{ksz}, $V[[x]]_E$ is identified as a model of ZFC for any $P_I$ generic $x$, under certain conditions on the ideal $I$.
The ideals which satisfy such conditions are those whose forcing $P_I$ adds a simple generic extension, in the sense that the intermediate extensions are well understood, such as Sacks and Silver forcing, and their iterations.

Our approach here goes in the other direction. We will see that for various equivalence relations and natural ideals on their domain (i.e., meager and null) the double brackets model fails to satisfy choice. Moreover, many properties of the equivalence relation can be understood by studying these choiceless models.

The focus in this paper is on equivalence relations which are classifiable by countable structures.
First we show that for such equivalence relations the double brackets model admits the following simple form.
For example, let $x\in\R^\omega$ be Cohen-generic over $V$ and $A=\set{x(n)}{n\in\omega}$ its $=^+$-invariant. 
In this case the double brackets model $V[[x]]_E$ is equal to the minimal model generated by $A$, $V(A)$, which is the ``basic Cohen model''. More generally:

\begin{prop}\label{prop;doublebrackets-symmetric-model}
Assume $x\mapsto A_x$ is a complete classification of $E$ by hereditarily countable structures.
Let $x$ be some generic in the domain of $E$, then $V[[x]]_E=V(A_x)$.
Moreover, there is a generic $x'Ex$ such that
\begin{equation*}
    V(A_x)=V[x]\cap V[x']=V[[x]]_E.
\end{equation*}
\end{prop}
The proof is based on the following claim, which gives a way to force a representative of an invariant $A$ over $V(A)$.
\begin{claim}\label{claim;forcing-rep}
Let $E$ and $x\mapsto A_x$ be as above. 
Let $x\in\dom E$ be a real in some generic extension and $A=A_x$.
Then there is a poset $P$ and $P$-name $\dot{g}$ in $V(A)$ such that in any $P$-generic extension of $V(A)$, $A_g=A$.
\end{claim}
\begin{proof}
Grigorieff \cite{Gri75} proved the following: for any generic extension $V[G]$ and any $A\in V[G]$, the model $V[G]$ is a generic extension of $V(A)$ (see also \cite{Zap01}).
Apply this to $A\in V[x]$. Let $P\in V(A)$ and $G$ a $P$-generic such that $V[x]=V[G]$. Fix a $P$-name $\dot{g}$ such that $\dot{g}[G]=x$.
Working over some condition which forces that $A_{\dot{g}} =A$, the claim follows. (Note that the statement $A_y =A$ is absolute, because this is a classification by countable structures.)

We note that finding the poset $P$ in our case is simple.
For example, in the basic Cohen model as above, one can force by finite approximations to add an enumeration $g$ of $A$.

\end{proof}

\begin{proof}[Proof of Proposition~\ref{prop;doublebrackets-symmetric-model}]
From the definitions, $V(A_x)\subset V[[x]]_E\subset V[x]\cap V[x']$, for any $x'Ex$.
It suffices to find some $x'$ in a generic extension such that $x'Ex$ and $V[x]\cap V[x'
]\subset V(A_x)$.
Let $A=A_x$ and $P$ be the poset from the claim.
Let $G$ be $P$-generic over $V[x]$ (hence also over $V(A)$) and $x'=\dot{g}[G]$.
By Lemma~\ref{lem;mutgen}, $V(A)[G]\cap V[x]=V(A)$.
In $V[x][x']$, $A_x=A=A_{x'}$ (by absoluteness of the map $y\mapsto A_y$), thus $xEx'$.
Furthermore, $V[x]\cap V[x']\subset V[x]\cap V(A)[G]=V(A)$.
\end{proof}

The conclusion that $V[[x]]_E$ can be represented as an intersection of only two models is true for all orbit equivalence relations. 
This can be used to prove a theorem of Kechris and Louveau \cite{KL97}, that $E_1$ is not reducible to an orbit equivalence relation (see \cite{Sha19}).

Next we establish the connection with Borel reducibility.

\begin{lem}\label{lem;reduction-doublebrackets}
Suppose $E$ and $F$ are Borel equivalence relations on $X$ and $Y$ respectively, and $f\colon X\lto Y$ is a Borel map.
Let $x\in X$ be some generic real.
\begin{itemize}
    \item If $f$ is a homomorphism, then $V[[f(x)]]_F\subset V[[x]]_E$;
    \item If $f$ is a reduction then $V[[x]]_E=V[[f(x)]]_F$.
\end{itemize}
\end{lem}
\begin{proof}
Assume $f$ is a homomorphism.
Let $x'$ be $E$-equivalent to $x$, in some further generic extension.
Then $f(x')Ff(x)$.
Thus
\begin{equation*}
V[[f(x)]]_F=V[[f(x')]]_F\subset V[f(x')]\subset V[x'],
\end{equation*}
for any $x'Ex$. 
It follows that $V[[f(x)]]_F\subset V[[x]]_E$.

Assume further that $f$ is a reduction.
Let $y$ be $F$-equivalent to $f(x)$ in some further generic extension.
By absoluteness for the statement $\exists x'(f(x')Fy)$, we can find such $x'$ in $V[y]$.
Since $f$ is a reduction, it follows that $x'Ex$.
As before
\begin{equation*}
    V[[x]]_E=V[[x']]_E\subset V[x']\subset V[y].
\end{equation*}
Since this is true for any $yFf(x)$, it follows that $V[[x]]_E\subset V[[f(x)]]_F$.
\end{proof}

More can be said when $E$ and $F$ are classifiable by countable structures.
\begin{lem}\label{lem;reduction-def-generating-set}
Suppose $E$ and $F$ are Borel equivalence relations on $X$ and $Y$ respectively and $x\mapsto A_x$ and $y\mapsto B_y$ are classifications by countable structures of $E$ and $F$ respectively.
Let $f\colon X\lto Y$ be a Borel map.
Suppose $x\in X$ is some generic real and let $A=A_x$ and $B=B_{f(x)}$.

\begin{itemize}
    \item If $f$ is a homomorphism, then $B\in V(A)$ is definable in $V(A)$ using only $A$ and parameters from $V$. That is, there is a formula $\psi$ and $v\in V$ such that $B$ is the unique set for which $\psi(B,A,v)$ holds.
    \item If $f$ is a reduction, then $V(A)=V(B)$.
    Furthermore, $A$ is definable using only $B$ and parameters from $V$.
\end{itemize}
\end{lem}
\begin{proof}
Assume that $f$ is a homomorphism.
$B$ can be defined in $V(A)$ as the unique set such for any generic $x'$, if $A_{x'}=A$ then $B=B_{f(x')}$. 

If $f$ is a reduction, 
$A$ can be defined in $V(B)$ as the unique set such that for any generic $x'$, if $B_{f(x')}=B$ then $A=A_{x'}$. 
\end{proof}

If $f$ is a partial Borel function, the conclusions of the lemmas above still hold, as long as the generic $x$ lies in the domain of $f$.
In particular, if $I$ is a proper ideal over $X$ (see \cite{Zap08}) and $f$ as above is only defined on some $I$-positive Borel set, then the conclusion holds for any $P_I$-generic $x$ in that set.
Here $P_I$ is the poset of all Borel $I$-positive sets, ordered by inclusion.
The converse is also true.

\begin{prop}\label{prop;reduction-from-sym-model}
Suppose $E$ and $F$ are as above and $I$ is a proper ideal over $X$.
Let $x$ be a $P_I$-generic, $A=A_x$.
Assume there is a $B\in V(A)$, definable in $V(A)$ using only $A$ and parameters from $V$, such that $B$ is an $F$-invariant in $V[x]$ (i.e.,  there is some $y\in V[x]$ such that $B_y=B$). Then
\begin{itemize}
    \item There is a partial Borel map $f\colon X\lto Y$, defined on an $I$-positive set, such that $f$ is a homomorphism, $x\in\dom f$, and $B=B_{f(x)}$.
    \item Furthermore, if $V(A)=V(B)$ and $A$ is definable using only $B$ and parameters from $V$, then there is an $f$ as above which is a partial reduction.
\end{itemize}
\end{prop}

\begin{proof}
Let $\phi$ be a formula and $v\in V$ such that $B$ is defined as the unique set satisfying $\phi(B,A,v)$ in $V(A)$.
Fix $y\in V[x]$ such that $B=B_y$, a name $\dot{y}$ for $y$ and a condition $p\in P_I$ forcing the above.
Fix a large enough countable model $M$. The set of $P_I$-generics over $M$ extending $p$ is an $I$-positive Borel set (see \cite{Zap08}, Proposition 2.2.2).
Let $f$ be defined on that set, sending $x$ to the interpretation of $\dot{y}$ under the generic corresponding to $x$.
Then $f$ is a partial Borel function defined on an $I$-positive set (see \cite{Zap08}).

Assume $x_1,x_2$ are both in the domain of $f$ and they are $E$-equivalent.
It follows that $A_{x_1}=A_{x_2}$. Let $A'=A_{x_1}=A_{x_2}$.
Since both $x_1$ and $x_2$ extend $p$, both sets $B_{f(x_1)}$ and $B_{f(x_2)}$ are defined in $M(A')$ as the unique $B'$ satisfying $\phi(B',A',v)$.
It follows that $B_{f(x_1)}=B_{f(x_2)}$ and therefore $f(x_1)Ff(x_2)$.
Thus $f$ is a homomorphism.

Assume now that $V(A)=V(B)$, and $A$ is definable using only $B$ and parameters from $V$.
Let $p$ be a condition forcing this and $f$ as above.
Assume $x_1$ and $x_2$ are in the domain of $f$ and $f(x_1),f(x_2)$ are $F$-related.
Let $B'\equiv B_{f(x_1)}=B_{f(x_2)}$.
Both $A_{x_1}$ and $A_{x_2}$ are defined in $M(B')$ using the same definition, from $B'$.
It follows that $A_{x_1}=A_{x_2}$ and so $x_1$ and $x_2$ are $E$-related.

\end{proof}

We summarize the correspondence between Borel reductions and definability in symmetric models in the following theorem.

\begin{thm}\label{thm;correspondence}
Suppose $E$ and $F$ are Borel equivalence relations on $X$ and $Y$ respectively and $x\mapsto A_x$ and $y\mapsto B_y$ are classifications by countable structures of $E$ and $F$ respectively.
Let $I$ be a proper ideal over $X$.
The following are equivalent:
\begin{itemize}
    \item There is a partial reduction of $E$ to $F$ defined on an $I$-positive Borel set;
    \item There is a $P_I$-generic $x$ and some $y\in V[x]$ such that $V(A_x)=V(B_y)$ and in this model $A_x$ and $B_y$ are definable from one another using only parameters from $V$.
\end{itemize}
\end{thm}

This correspondence will be used below in the following way.
To show that $E$ is not Borel reducible to $F$, we find some generic $E$-invariant $A$ such that the model $V(A)$ is not of the form $V(B)$ whenever $B\in V(A)$ is an $F$-invariant which is definable using only $A$ and parameters from $V$.
The invariant $A$ will often be of the form $A_x$ where $x$ is a Cohen generic, with respect to some topology.
In this case it follows that there is no partial reduction on any non-meager set (see Section~\ref{sec;topologies}).

\begin{remark}
This paper deals with the Friedman-Stanley jumps and the equivalence relations of Hjorth-Kechris-Louveau.
Applications of the techniques above to study equivalence relations below $=^+$ will appear in \cite{Sha18}
\end{remark}

\section{Finite jumps and the models of G. Monro}\label{sec;Monro-models}

In this section we present a proof that $=^{+(n+1)}$ is not reducible to $=^{+n}$ using the techniques developed above.
To that end, we need to find models $M_n$ such that $M_n$ is generated by a $=^{+n}$-invariant, that is, a set in $\mathcal{P}^{n+1}(\omega)$, but not generated by any set in $\mathcal{P}^m(\omega)$ for $m\leq n$.
It turns out that the right models were constructed by Monro \cite{Mon73} in order to separate the finite generalized Kinna-Wagner principles.
First we describe Monro's construction and make a few remarks about the relationship between the Kinna-Wagner principles and the question of interest here, of which type of sets can generate the model.
Finally we strengthen Monro's analysis and deduce the Borel irreducibility results.

\begin{defn}
For a set $A$, define $P(A)$ to be the poset of all finite functions $p\colon\dom p\lto 2$, where $\dom p\subset A\times A$, ordered by extension.
\end{defn}

Forcing with $P(A)$ adds a function $g\colon A\times A\lto \{0,1\}$.
Let $A^0=\omega$, $M_0=V$.
Given $M_n$, $A^n\in M_n$, let $g\colon A^n\times A^n\lto\{0,1\}$ be $P(A^n)$-generic over $M_n$. Define
\begin{equation*}
    A^{n+1}_a=\set{b\in A^n}{g(a,b)=1},\textrm{ for } a\in A^n,\, A^{n+1}=\set{A^{n+1}_a}{a\in A^n},
\end{equation*}
and $M_{n+1}=M_n(A^{n+1})$. We will consider $G_n=\seqq{A^{n+1}_a}{a\in A^n}$, a collection of subset of $A^n$ indexed by $A^n$, as the $P(A^n)$-generic object.
For example, $P(A^0)=P(\omega)$ is simply Cohen forcing for adding countably many Cohen reals.
$G_0=\seqq{A^1_n}{n<\omega}$ is a generic sequence of Cohen reals, $A^1$ is the unordered collection and $M_1=V(A^1)$ is the basic Cohen model.

Monro \cite{Mon73} shows that in $M_{n}$, $\mathrm{KWP}^{n-1}$ fails yet $\mathrm{KWP^{n+1}}$ holds.
A more careful analysis shows that in fact each $M_n$ satisfies $\mathrm{KWP}^n$.
Note that $A^m$ is definable from $A^n$ for $m<n$, hence $M_n=V(A^n)$ for all $n$.

\begin{defn}
For an ordinal $\alpha$ let $\mathcal{P}^\alpha(\mathrm{On})$ be the class of all sets in $\mathcal{P}^{\alpha}(\eta)$ for some ordinal $\eta$.
Say that a set is of {\bf rank $\mathbf{\alpha}$} if it is in $\mathcal{P}^\alpha(\mathrm{On})$.
\end{defn}
\begin{obs}\label{obs;KW-implies-low-generator}
Suppose $M=V(A)$ where $A$ is of rank $n+1$, and there is an injective function $f\colon A\lto B$ where $B$ is of rank $n$.
Then $M=V(C)$ for some set $C$ of rank $n$.
In particular, if a model $M$ is generated by a set of rank $n+1$, but not by a set of rank $n$, then $\mathrm{KWP}^{n-1}$ fails in $M$.
\end{obs}
\begin{proof}
Define $C=\set{(x,y)}{\exists X\in A(y=f(X)\wedge x\in X)}$.
$C$ is a set of pairs of rank $n-1$ sets, thus can be coded by a rank $n$ set.
Furthermore, $A$ is definable from $C$, thus $M=V(A)=V(C)$.
\end{proof}
The reverse implication does not hold.
If $x\in\R^{\omega^2}$ is a Cohen-generic and $A$ is the $=^{++}$-invariant of $x$ (a set of rank $3$), then $V(A)$ is generated by a set of reals (see Proposition~\ref{prop;reduction-of-++}).
However, $\mathrm{KWP}^1$ fails in $V(A)$. 

The failure of Kinna-Wagner principles is the crucial property of Monro's models.
In Monro's models and the generalizations we construct below we will mention which Kinna-Wagner principle holds without proof, as this is not needed for our applications.
Given the analysis of the models, the proofs are analogous to the proof that $\mathrm{KWP}^1$ holds in the basic Cohen model (see \cite{HL64}, also \cite{Jec73} and \cite{Kar16}).

Monro's proof that $\mathrm{KWP}^{n-1}$ fails in $M_n$ relies on the following lemma (which we reprove in Section~\ref{subsec;Monro-posets-tech}).

\begin{lem}[Monro \cite{Mon73}, Theorem 8]\label{lem;Monro-thm8}
For any ordinal $\eta$, $M_{n+1}\cap\mathcal{P}^{n}(\eta)=M_{n}\cap\mathcal{P}^{n}(\eta)$.
That is, $M_n$ and $M_{n+1}$ have the same sets of rank $n$.
\end{lem}

Monro concludes by observing that the existence of a model $M_{n+1}$ which has the same sets of rank $n$ as $M_n$ yet is different than $M_n$, implies that $\mathrm{KWP}^{n-1}$ must fail in $M_n$ (Theorem 3 in \cite{Mon73}).
This is a direct generalization of the theorem of Vopenka and Balcar, which states that for two models of ZF, one of which satisfies choice, if they agree on sets of ordinals, then they are the same.

\begin{lem}\label{lem;gen-blass-thm}
Suppose $B\in V(A^{n+1})\cap\mathcal{P}^{n+1}(\eta)$ for some ordinal $\eta$.
Then there is a finite $\bar{a}\subset A^{n+1}$ such that $B\in V(A^n)(\bar{a})$.
\end{lem}
\begin{proof}
By Lemma~\ref{lem;Monro-thm8}, $B\subset V(A^n)$, so the conclusion follows from Lemma~\ref{lem;symmetry}.
\end{proof}

It follows that $M_n$ is not generated (over $V$) by a set of rank $n$. This implies the failure of $\mathrm{KWP}^{n-1}$ by Observation~\ref{obs;KW-implies-low-generator}.

\begin{cor}[Friedman-Stanley \cite{FS89}]\label{cor;Friedman-Stanley}
$=^{+(n+1)}$ is not Borel reducible to $=^{+n}$.
\end{cor}
\begin{proof}
Let $x\mapsto A_x$ and $y\mapsto B_y$ be the natural classifications of $=^{+(n+1)}$ and $=^{+n}$ with invariants in $\mathcal{P}^{n+2}(\omega)$ and $\mathcal{P}^{n+1}(\omega)$ respectively.
We claim that in a generic extension of $V(A^{n+1})$ there is some $x$ such that $A^{n+1}=A_x$, this is justified below. Given this claim, assume for contradiction that $f$ is a reduction from $=^{+(n+1)}$ to $=^{+n}$, let $B=B_{f(x)}$.
By Lemma~\ref{lem;reduction-def-generating-set}, $V(A^{n+1})=V(B)$, where $B\in\mathcal{P}^{n+1}(\omega)$. However, by Lemma~\ref{lem;gen-blass-thm}, $B\in V(A^n)(\bar{a})$ for some finite $\bar{a}\subset A^{n+1}$. Note that $V(A^n)(\bar{a})\subsetneq V(A^{n+1})$, as any element in $A^{n+1}\setminus \bar{a}$ is generic over $V(A^n)(\bar{a})$. It follows that $V(B)\subsetneq V(A^{n+1})$, a contradiction.

We now justify that in some generic extension of $V(A^{n+1})$ there is an $x$ in the domain of $=^{+(n+1)}$ such that $A=A_x$.
Note that in any model of ZFC, if $A\in\mathcal{P}^{n+2}(\omega)$ is hereditarily countable, then there is some $x$ such that $A=A_x$. This is because in the natural complete classification of $=^{+(n+1)}$ the invariants are all hereditarily countable sets in $\mathcal{P}^{n+2}(\omega)$ (see \cite{FS89}, \cite{HKL98}).
Now we may force over $V(A^{n+1})$ to collapse $A^{n+1}$ and its transitive closure, and regain choice. Thus in this collapse there will be an $x$ as desired.
Alternatively, the model $V(A^{n+1})$ can be seen as a submodel of a single generic extension of $V$. (This is done explicitly for $=^{+2}$ in Section~\ref{sec;topologies}. In Section~\ref{subsec;cong-ast-omega+1} we present a generic extension of $V$ (by the poset $P^{0,\omega}$) containing $V(A^m)$ for all $m$.)
It then follows from Grigorieff's result that choice can be forced over $V(A^{n+1})$ (see the proof of Claim~\ref{claim;forcing-rep}).

\end{proof}

\subsection{More on Monro's models}\label{subsec;Monro-posets-tech}
We prove generalizations of some lemmas from \cite{Mon73}.
This will be necessary for section~\ref{sec;cong-ast-n-k} below.
Monro's arguments are based on the following lemma.

\begin{lem}[Monro \cite{Mon73}, Lemma 6]\label{lem;Monro-lem6}
Let $\psi$ be some formula, $x\in M_{k-1}$.
Assume $\{r_1,...,r_n\}$, $\{s_1,...,s_m\}$ are disjoint subsets of $A_k$ and $\psi(A_k,r_1,...,r_n,x,s_1,...,s_m)$ hold in $M_k$.
Then there are finite functions $f_i\colon \dom f_i\lto 2$, $\dom f_i\subset A_{k-1}$, $1\leq i\leq m$ such that for any $t_1,...,t_m$ from $A_k$, if $t_i\supset f_i$ then $\psi(A_k,r_1,...,r_n,x,t_1,...,t_m)$ holds.
\end{lem}

Work in $M_n$.
Let $Z$ be a subset of $M_{n-1}$ and consider the poset $P$ of all finite functions $p\colon \dom p\lto Z$ where $\dom p\subset A^n\times A^n$.
A generic gives a function $g\colon A^n\times A^n\lto Z$.
Taking $Z=\{0,1\}$ we get Monro's poset $P(A^n)$ as above. 
\begin{lem}[Strengthening of Lemma 7 in \cite{Mon73}]\label{lem;strong-Monro-7}
Let $\psi$ be a formula, $p\in P$, $r_1,...,r_m$ in $A^n$ and $x\in M_{n-1}$ such that in $M_n$
\begin{equation*}
    p\force \psi(A^n,r_1,...,r_m,x).
\end{equation*}
Then 
\begin{equation*}
    p\rest{\{r_1,...,r_m\}^2}\force \psi(A^n,r_1,...,r_m,x).
\end{equation*}
\end{lem}
\begin{proof}
We show that any condition $q$ extending $p\rest{\{r_1,...,r_m\}^2}$ is compatible with some condition forcing $\psi(A^n,r_1,...,r_m,x)$.

Take $s_1,...,s_k$ in $A^n$, disjoint from $r_1,...,r_m$, such that the domain of $p$ is included in $\{r_1,...,r_m,s_1,...,s_k\}^2$.
Let $y\in M_{n-1}$ be some parameter coding $x$ and the image of $p$. (Recall that the image of $p$ is in $Z$ and therefore in $M_{n-1}$).

Let $\phi(A^n,r_1,...,r_m,s_1,...,s_k,y)$ be the formula postulating that $p\force\psi(A^n,r_1,...,r_m,x)$.
By Lemma~\ref{lem;Monro-lem6} applied to $\phi$, there are finite functions $f_1,...,f_k$, $\dom f_i\subset A^{n-1}$, $f_i\colon \dom f_i\lto 2$ such that for any $t_1,...,t_k$ in $A^n$ with $t_i\supset f_i$, $\phi(A^n,r_1,...,r_m,t_1,...,t_k,y)$ holds.
That is, $p[t_1,...,t_k]\force \psi(A^n,r_1,...,r_m,x)$, where $p[t_1,...,t_k]$ is defined by replacing $t_i$ with $s_i$ in $p$.

Finally, for any $q\leq p\rest \{r_1,...,r_m\}^2$ we can find some $t_1,...,t_k$ disjoint from the domain of $q$, with $t_i\supset f_i$, and extend $q$ to $q'$ such that $q'\rest\{r_1,...,r_m,t_1,..,t_k\}^2=p[t_1,...,t_k]$.
It follows that $q'\force \psi(A^n,r_1,...,r_m,x)$.
This finishes the proof.
\end{proof}

\begin{lem}[Strengthening of Theorem 8 in \cite{Mon73}]\label{lem;stong-Monro-8}
No sets of rank $n$ are added by forcing with $P$ over $M_n$.
\end{lem}
\begin{proof}
Note that the lemma generalizes Lemma~\ref{lem;Monro-thm8}. We prove both lemmas simultanoeusly by induction on $n$.
Assuming Lemma~\ref{lem;Monro-thm8} for $M_{n-1}$, we prove by induction on $j\leq n$ that no new sets of rank $j$ are added to $M_{n}$ by forcing with $P$.
Assume the result for $j<n$, and fix a name $\dot{B}$ for a rank $j+1$ set.
It remains to show that $\dot{B}$ is forced to be in $M_n$.

By the inductive hypothesis, $B$ is a subset of $M_{n-1}$.
Since $\dot{B}\in M_n$, there is a formula $\phi$, finitely many parameters $r_1,...,r_m$ from $A^n$ and $v\in M_{n-1}$ such that $\dot{B}$ is defined by $\phi(\dot{B},A^n,r_1,...,r_m,v)$.
Suppose $p\in P$ and $x$ of rank $j$ is such that $p\force \check{x}\in \dot{B}$.
The statement $\check{x}\in \dot{B}$ involves $A^n$, $r_1,...,r_m$, $x$ and $v$ (which are both in $M_{n-1}$).
By the lemma above, $p\rest \{r_1,...,r_m\}^2$ forces that $\check{x}\in\dot{B}$.
It follows that any $p$ whose domain include $\{r_1,...,r_m\}^2$ decides all elements of $\dot{B}$, hence $\dot{B}$ is forced to be in the ground model $M_{n}$. 

\end{proof}

\section{Generic behaviour of the Friedman-Stanley jumps}\label{sec;topologies}

Consider the product measure and product topology on $\R^{\omega^2}$.
Define $D\subset \R^{\omega^2}$ to be the set of all elements $x\in\R^{\omega^2}$ such that all the corresponding $\omega^2$-many reals are different.
That is, for any distinct $a,b\in\omega^2$, $x(a)\neq x(b)$.
Note that $D$ is comeager and conull.

\begin{prop}\label{prop;reduction-of-++}
$=^{++}\rest D$ is Borel reducible to $=^+$.
\end{prop}
\begin{proof}
Let us argue in terms of invariants.
For $x\in D$, its $=^{++}$ invariant is a set of sets of reals $A$, such that any distinct $X,Y\in A$ are disjoint.
Define $R=\set{(x,y)}{\exists X\in A(x,y\in X)}$.
$R$ is the equivalence relation partitioning the set of reals $\bigcup A$ to $A$.
Thus $A$ is defined from $R$, and $R$ can be coded as a set of reals, hence an $=^+$ invariant.

A Borel map $\R^{\omega^2}\lto (\R^2)^\omega$ can be defined, sending an element of $\R^{\omega^2}$ coding $A$ to a sequence of pairs of reals coding $R$, and is a reduction of $=^{++}\rest D$ to $=^+$.
\end{proof}

In particular, if $x\in\R^{\omega^2}$ is either Cohen or random generic and $A$ is its $=^{++}$-invariant, then $V(A)$ is generated by a set of reals.
From Monro's model $V(A^2)$ we get the following presentation of $=^{++}$.
Given $x=(x_1,x_2)\in\R^\omega\times(2^\omega)^\omega$, let
\begin{equation*}
    A^i_x=\set{x_1(j)}{x_2(i)(j)=1}\textrm{ and } A_x=\set{A^i_x}{i\in\omega}.     
\end{equation*}
Define $F$ on $\R^\omega\times(2^\omega)^\omega$ by $xFy$ if and only if $A_x=A_y$.
Then $F$ is Borel bireducible with $=^{++}$.

\begin{prop}\label{prop;=++-no-reduction}
 For any non meager set $C$ in the standard product topology on $\R^\omega\times(2^\omega)^\omega$, $F\rest C$ is \textit{not} Borel reducible to $=^+$.
\end{prop}
\begin{proof}
By Theorem~\ref{thm;correspondence} it suffices to show that given a Cohen-generic $x=(x_1,x_2)\in\R^\omega\times(2^\omega)^\omega$, $V(A_x)$ is not generated by a set of reals.

Let $A^1=\set{x_1(i)}{i\in\omega}$ and $A^2=A_x$.
Then $A^1$ is a Cohen set and $A^2$ is a set of mutually generic subsets of $A^1$ over $V(A^1)$ (this follows from the presentation of Monro's iterations as $P^{0,\omega}$ (Definition~\ref{def;monro-iterations}) and Lemma~\ref{lem;monro-iterations-projections} below).
By Lemma~\ref{lem;gen-blass-thm}, $V(A^2)$ is not generated by a set of reals: for any set of reals $B\in V(A^{2})$, $B\in V(A^1)(\bar{a})$ where $\bar{a}\subset A^2$ is finite, so the sets in $A^2\setminus \bar{a}$ are generic over $V(A^1)(\bar{a})$.
\end{proof}

Define the union map $u\colon \R^\omega\times(2^\omega)^\omega\lto \R^\omega$ by
\begin{equation*}
    u(x_1,x_2)=x_1.
\end{equation*}
Let us work in the comeager subset of $\mathbb{R}^\omega\times(2^\omega)^\omega$ where $\forall j\exists i(z(i)(j)=1)$. That is, each real $x_1(j)$ appears in one of the sets $A^i_x$. In this case $u$ is a homomorphism from $F$ to $=^+$. Furthermore, $u$ maps the $F$-invariant $A_{x}=\set{A_x^i}{i\in\omega}$ to its union, the $=^+$-invariant $\bigcup A_{x}=\set{x(j)}{j\in\omega}$.

\begin{prop}\label{prop;=++-unique-homo}
Suppose $f\colon \R^\omega\times(2^\omega)^\omega\lto\R^\omega$ is a homomorphism from $F$ to $=^+$.
Then there is a homomorphism $h$ from $=^+$ to $=^+$ defined on a comeager set such that $f=^+ h\circ u$ (that is, $f(x)\mathrel{=^+}h(u(x))$ for comeagerly many $x$).
\end{prop}
\begin{proof}
Let $x=(x_1,x_2)\in \R^\omega\times(2^\omega)^\omega$ be Cohen generic, $A^2=A_x$ and $A^1=\bigcup A^2=\set{x_1(i)}{i\in\omega}$.
Let $y=f(x)\in\R^\omega$ and $B=\set{y(i)}{i\in\omega}$ the corresponding $=^+$-invariant.
By Lemma~\ref{lem;reduction-def-generating-set} $B$ is definable in $V(A^2)$ from $A^2$ and parameters in $V$.
By Lemma~\ref{lem;symmetry} it follows that $B$ is in $V(A^1)$ and is definable only from $A^1$ and parameters from $V$.

Note that $x_1\in\R^\omega$ is Cohen-generic and $A^1$ is its $=^+$-invariant.
According to Proposition~\ref{prop;reduction-from-sym-model} this corresponds to a partial Borel homomorphism $h\colon =^+ \lto =^+$ defined on a non-meager subset of $\R^\omega$ such that $h(x_1)=y$.
In this case $h$ can be defined on a comeager set, as the statements in Proposition~\ref{prop;reduction-from-sym-model} are forced by the empty condition.
It also follows from the definition of $h$ in Proposition~\ref{prop;reduction-from-sym-model} that $f(x)=^+ h\circ u(x)$ on a comeager set. 
\end{proof}

\section{Generic invariants for $\cong^\ast_{n,k}$}\label{sec;cong-ast-n-k} 
In this section we prove part (1) of Conjecture~\ref{conj;HKL}, that $\cong^\ast_{n,k}$ is not Borel reducible to $\cong^\ast_{n,k-1}$, when defined.
To show such irreducibility result we study models of the form $V(A)$ where $A$ is an invariant for $\cong^\ast_{n,k}$, and we need to show it is not of the form $V(B)$ where $B$ is an invariant for $\cong^\ast_{n,k-1}$.
First we review what are invariants for the equivalence relations of Hjorth-Kechris-Louveau and make a few simplifications.
Corollary~\ref{cor;ast-n-k-invariants} below gives a simple condition sufficient to establish that a model is generated by an invariant for $\cong^\ast_{n,k}$.
Claim~\ref{claim;cong-ast-invs} establishes a property of models generated by invariants for $\cong^\ast_{n,k}$.
Section~\ref{subsec;cong-ast-31} provides a ``hands-on'' proof that $\cong^\ast_{3,1}$ is not Borel reducible to $\cong^\ast_{3,0}$, to illustrate our techniques. The proof of Conjecture~\ref{conj;HKL} part (1) is in Section~\ref{subsec;cong-ast-n-k}.

\subsection{Invariants for the Hjorth-Kechris-Louveau equivalence relations}
Recall that an invariant for $\cong^\ast_{\alpha+1,\beta}$ is of the form $(A,R)$, where $A$ is a set of rank $\alpha+1$ in $\mathcal{P}^{\alpha+1}(\omega)$ and $R$ is a relations on $A\times A\times( \mathcal{P}^{\beta}(\omega)\cap\mathrm{tc}(A))$ such that for any $a,b,b'\in A$ and $x\in \mathcal{P}^\beta(\omega)\cap \mathrm{tc}(A)$, if $R(a,b,x)$ and $R(a,b',x)$ holds, then $b=b'$.
In other words, given $a\in A$, the map $b\mapsto \set{x}{R(a,b,x)}$ maps the members of $A$ into disjoint sets of rank $\beta+1$.
(See Section~\ref{sec;HKL-relations}).

When $\beta=0$, the relation $R$ allows us to enumerate the invariant $A$, uniformly in a parameter from $A$.
This behaviour is analogous to countable equivalence relations. 
Given a countable equivalence relation $E$ and an $E$-class $A=[x]_E$ (a set of rank 2), this class can be enumerated using any parameter from $[x]_E$.

In the generic invariants that we construct below, the relation $R$ will have the stronger property that for any $a,b\in A$ there is a unique $c$ such that $R(a,b,c)$ holds. In other words, for any $a\in A$, $R(a,-,-)$ defines an injective map from $A$ into $\mathcal{P}^\beta(\omega)$ (coding $A$ as a set of lower rank).

\begin{remark}\label{remark;cong-ast-defn}
The definition of the sets $\mathcal{P}^\alpha(\omega)$ of rank $\alpha$ in \cite{HKL98} is slightly different than we use here, and is defined by $\mathcal{P}^{\alpha}(\omega)=\mathcal{P}(\mathcal{P}^{<\alpha}(\omega)\cup\omega)$.
This is used to fix a particular coding of finite sequences of rank $\beta$ sets as rank $\beta$ sets (see \cite[p.71]{HKL98}).

A similar coding will be used here as well, just in our context we work often with unordered finite subsets.
(Working in symmetric models, a set of rank $3$ may not even admit a linear order.)
To that end we fix injective maps (working in ZF) between $[\mathcal{P}^k(\omega)]^{<\aleph_0}$ and $\mathcal{P}^{k}(\omega)$ as follows.
\end{remark}

Fix some injective map $f_1$ from $[\mathcal{P}^1(\omega)]^{<\aleph_0}$ (finite sets of reals) and $\mathcal{P}^1(\omega)$.
This can be done by using the linear ordering of the reals to code finite subsets as finite sequences.
Given $f_k\colon [\mathcal{P}^k(\omega)]^{<\aleph_0}\lto\mathcal{P}^k(\omega)$ define $f_{k+1}\colon [\mathcal{P}^{k+1}(\omega)]^{<\aleph_0}\lto\mathcal{P}^{k+1}(\omega)$ by
\begin{equation*}
    f_{k+1}(X)=\set{f_k(\set{c(x)}{x\in X})}{c\textrm{ is a choice function for } X}.
\end{equation*}
\begin{claim}
For each $k$, $f_k$ is injective.
\end{claim}
This can be extended to all countable ordinals in a similar way.
In the examples below our invariants will be $(A,R)$ such that $R\subset A\times A\times [\mathcal{P}^k(\omega)\cap \mathrm{tc}(A)]^{<\aleph_0}$.
We want to use the coding function $f_k$ above to conclude that $(A,R)$ is a $\cong^\ast_{n,k}$-invariant.

Another minor detail is that after composing with $f_k$ the low rank members of $R$ may no longer be in the transitive closure of $A$.
This is not a real issue as they are still definable from $\mathrm{tc}(A)$ in a simple way.
For example, by changing $A$ a little one can find a pair $(\tilde{A},\tilde{R})$ which is bi-definable with $(A,R)$ and satisfies the conditions of being a $\cong^\ast_{n,k}$-invariant.

The main task ahead is to find ``good'' generic invariants for the equivalence relations $\cong^\ast_{n,k}$.
In all the examples below the relation $R$ will be definable from the set $A$ in a simple way.
When defining invariants we will always rely on the following conclusion of the above discussion.

\begin{cor}\label{cor;ast-n-k-invariants}
To find a $\cong^\ast_{n,k}$-invariant it suffices to find a set $A$ of rank $n$ in $\mathcal{P}^n(\omega)$ such that there are injective functions, definable uniformly from $A$ and a parameter from $A$, sending the members of $A$ to finite subsets of $\mathrm{tc}(A)\cap \mathcal{P}^k(\omega)$.
\end{cor}
To show that a given model is \textit{not} generated by a $\cong^\ast_{\alpha+1,\beta}$-invariant, the following observation will be repeatedly used.

\begin{claim}\label{claim;cong-ast-invs}
Let $(A,R)$ be a generic $\cong^\ast_{\alpha+1,\beta}$-invariant.
Then $V(A,R)$ can be written as $V(B)$ where $B$ is definable using only $(A,R)$ and parameters from $V$, $B$ is of rank $\alpha+1$ and
there is a relation $\tilde{R}\subset A\times B\times\mathcal{P}^\beta(\omega)$ such that for any $a\in A$, $b,b'\in B$ and $x\in \mathcal{P}^\beta(\omega)$, if $\tilde{R}(a,b,x)$ and $\tilde{R}(a,b',x)$ hold then $b=b'$. Furthermore, for any $a\in A$ and $b\in B$ there is some $x$ such that $\tilde{R}(a,b,x)$ holds.
\end{claim}
In particular, given any $a\in A$ then map $b\mapsto \set{x}{\tilde{R}(a,b,x)}$ maps $B$ into a set of {\it disjoint} subsets of $\mathcal{P}^{\beta}(\omega)$.
If $\beta=0$, we get injective maps from $B$ into disjoint subsets of $\omega$, so $B$ is countable in $V(B)=V(A,R)$.
\begin{proof}
By assumption, $R$ is a relation such that for any $a,b,b'\in A$ and $x\in \mathcal{P}^\beta(\omega)$, if $R(a,b,x)$ and $R(a,b',x)$ then $b=b'$, and such that for any $a,b\in A$ there is an $x$ for which $R(a,b,x)$ holds.
So the properties of $B$ are true if we replace $B$ with $A$. 
We extend this to the pair $(A,R)$ and then take $B$ to code $(A,R)$.

Define a relation $R'\subset A\times A\times R\times (\mathcal{P}^\beta(\omega))^4$ by
\begin{equation*}
R'(a,b,(c,d,u),(w,x,y,z)) \iff R(a,b,w)\wedge R(a,c,x) \wedge R(a,d,y) \wedge u=z.     
\end{equation*}
Just like $R$, $R'$ satisfies that: for any $a\in A$, for any $(b,c,d,u),(b',c',d',u')\in A\times R$ and any $(w,x,y,z)\in (\mathcal{P}^\beta(\omega))^4$, if
\begin{equation*}
    R'(a,b,(c,d,u),(w,x,y,z)) \textrm{ and } R'(a,b',(c',d',u'),(w,x,y,z)),
\end{equation*}
then $b=b'$ and $(c,d,u)=(c',d',u')$.

Fix a definable injective map $\gamma$ from $(\mathcal{P}^\beta(\omega))^4$ into $\mathcal{P}^\beta(\omega)$.
Fix a definable injective map $\delta$ from $\mathcal{P}^{\alpha+1}(\omega)\times \mathcal{P}^{\alpha+1}(\omega)\times \mathcal{P}^{\alpha+1}(\omega)\times \mathcal{P}^{\beta+1}(\omega)$ into $\mathcal{P}^{\alpha+1}(\omega)$.
Let $B$ be a set of rank $\alpha+1$ coding $(A,R)$ via $\delta$.

Finally, define $\tilde{R}\subset A\times B\times \mathcal{P}^\beta(\omega)$ by $\tilde{R}(a,e,v)$ if and only if $e=\delta(b,c,d,u)$, $v=\gamma(w,x,y,z)$ and $R'(a,b,(c,d,u),(w,x,y,z))$.
\end{proof}
Note that if $\alpha=\beta+1$, then the claim above provides a set $B$ of rank $\alpha+1$ with an embedding into a set of rank $\beta+2=\alpha+1$. However, the disjointness in the image of the embedding provides the following significant simplification (as in Proposition~5.1).

Suppose $X$ is a set of disjoint subsets of $\mathcal{P}^\beta(\omega)$ ($X\in\mathcal{P}^{\beta+2}(\omega)$). Then there is $Y\in\mathcal{P}^{\beta+1}(\omega)$, definable from $X$, such that $V(X)=V(Y)$.

We argue as in Proposition~5.1.
Let $Z=\bigcup X\in\mathcal{P}^{\beta+1}(\omega)$. 
Let $W$ be the equivalence relation on $Z$ defined by $x\mathrel{W} y$ if and only if there is $D\in X$ such that $x,y\in D$.
Since $Z$ is of rank $\beta+1$, $W$ can be coded as a set of rank $\beta+1$ as above.
Let $Y$ be a set of rank $\beta+1$ coding $(Z,W)$.
Now $V(X)=V(Z,W)=V(Y)$.

\subsection{An interesting $\cong^\ast_{3,1}$-invariant}\label{subsec;cong-ast-31}

In this section we describe a simple proof of the fact that $\cong^\ast_{3,1}$ is not Borel reducible to $\cong^\ast_{3,0}$.

Work in the Cohen model $V(A^1)$ (using the notation from Monro's construction above).
Let $P$ be the poset of all finite functions $p\colon\dom p\lto 2$ where $\dom p$ is a subset of $A^1$.
$P$ adds a single generic subset of $A^1$, and is a sub forcing of $P(A^1)$ above.
It follows from Lemma~\ref{lem;Monro-thm8} that forcing with $P$ adds no reals.

Let $X\subset A^1$ be $P$-generic over $V(A^1)$. For the remaining of Section~\ref{subsec;cong-ast-31} we work in $V(A^1)[X]$, a generic extension of the basic Cohen model.
Define 
\begin{equation*}
   A=\set{X\Delta a}{a\subset A^1\textrm{ is finite}}. 
\end{equation*}
$A$ is a set of subsets of $A^1$, containing $X$ and all of its finite alterations.
Note that for any $Y,Z\in A$, $Y\Delta Z$ is a finite subset of $A^1$ (which we consider as a real).
\begin{equation*}
    \textrm{Given }Y\in A\textrm{, the map }Z\mapsto Z\Delta Y\textrm{ is injective, from }A\textrm{ to reals}.
\end{equation*}
We conclude that $A$ is a $\cong^\ast_{3,1}$-invariant 
(see Corollary~\ref{cor;ast-n-k-invariants}). 

\begin{claim}
$\cong^\ast_{3,1}$ is not Borel reducible to $\cong^\ast_{3,0}$.
\end{claim}
To prove this irreducibility result we study the model $V(A)$ generated by the $\cong^\ast_{3,1}$-invariant $A$.
By Lemma~\ref{lem;reduction-def-generating-set}, to prove that $\cong^\ast_{3,1}$ is not Borel reducible to $\cong^\ast_{3,0}$ it suffices to show that $V(A)$ is not of the form $V(B)$ for any $\cong^\ast_{3,0}$-invariant $B$ which is definable from $A$.
By Claim~\ref{claim;cong-ast-invs} it suffices to prove the following:

\begin{prop}\label{prop;cong31-not-30}
In $V(A)$: let $B$ be a set of sets of reals, definable from $A$ and parameters in $V$. Assume further that $B$ is countable.
Then $V(B)\subsetneq V(A)$.
\end{prop}
Recall that $A$ was defined in $V(A^1)[X]$, where $X$ is a subset of $A^1$, generic over $V(A^1)$. Since $X\in A$ and $A^1=\bigcup A$, it follows that $V(A^1)[X]=V(A)$. 
In particular, for $B=\{A^1,X\}$, a set of two sets of reals, $V(A)=V(B)$. The point is that $X$ is not definable from $A$.

\begin{proof}[Proof of Proposition~\ref{prop;cong31-not-30}]
Assume towards contradiction that $V(B)=V(A)$ where $B\in V(A)$ is a countable set of sets of reals which is definable from $A$ and parameters in $V$ alone. 
Since $X\in V(A)$, it follows that $X\in V(B)$ and therefore $X$ is definable from $B$ and its transitive closure (see Section~\ref{sec;prelm}, before Remark~\ref{remark-symmetricm-presentation}). That is, there is a formula $\psi$, finitely many parameters $U_1,..,U_k\in B$ and a real $z$ such that $X$ is the unique set satisfying $\psi(X,B,U_1,...,U_k,z)$. (Any finitely many reals from the transitive closure of $B$ can be coded by a single real.)
Since $V(A)$ and $V(A^1)$ agree on reals, $z\in V(A^1)$.

Fix some condition $r\in P$ forcing the above and work in $V(A^1)$. 
For any $a\in A^1\setminus\dom r$, let $\pi_a$ be the permutation of $P$ swapping the value of $a$. 
Then $\pi_ar=r$ and $\pi_a \dot{A}=\dot{A}$.
Since $B$ is definable from $A$ and parameters in $V$, it follows that $\pi_a \dot{B}=\dot{B}$ as well.
In particular, for any such $a$, $r=\pi_a r$ forces that $\pi_a\dot{U}_j\in \dot{B}$ and $\dot{X}\Delta\{a\}=\pi_a \dot{X}$ is defined uniquely by $\psi(\dot{X}\Delta\{a\},\dot{B},\pi_a\dot{U}_1,...,\pi_a\dot{U}_k,\check{z})$.

Thus the map sending $a$ to $\pi_a U_1,...,\pi_a U_k$ is injective between $A^1\setminus\dom r$ and $B^k$.
Since $B$ is countable, so is $B^k$ and therefore $A^1$ is countable.
Since $P$ adds no reals, this is a contradiction.
\end{proof}

\subsection{The general case}\label{subsec;cong-ast-n-k}

Fix $n\geq 2$ and $1\leq k\leq n$.
To find a good invariant for $\cong^\ast_{n+2,k}$, we want to take an orbit of a rank $n+1$ set under an action of $[A^k]^{<\aleph_0}$.
In the example above $[A^1]^{<\aleph_0}$ was acting on the subsets of $A^1$ (a set of rank $2$) by symmetric differences.
However, for $k$ much smaller than $n$, there are no nontrivial actions of $A^k$ on $A^n$ (see Lemma~\ref{lem;indiscern} below).
We will add a non trivial action by forcing a generic function $g\colon A^n\lto A^k$.

Let $P$ in $V(A^n)$ be the poset of all finite functions $p\colon \dom p\lto A^k$ where $\dom p\subset A^{n}$.
Let $g\colon A^{n}\lto A^k$ be $P$-generic over $V(A^n)$.
By Lemma~\ref{lem;stong-Monro-8} $P$ adds no rank $n$ sets over $V(A^n)$.
Work in $V(A^{n})[g]=V(g)$.
Let $\Pi$ be the group of all finite support permutations of $A^k$.
The members of $\Pi$ are coded by finite subsets of pairs from $A^k$, and so are of rank $k$, according to the discussion above.
Define 
\begin{equation*}
    A=\set{\pi\circ g}{\pi\in\Pi}.
\end{equation*}
Each $\pi\circ g$ is a set of rank $n+1$, thus $A$ is of rank $n+2$.
Given any $h\in A$, the map $\pi\mapsto \pi\circ h$ is a bijection between $\Pi$ and $A$.
It follows that $A$ is a $\cong^\ast_{n+2,k}$-invariant (see Corollary~\ref{cor;ast-n-k-invariants}).
Note that $V(g)=V(A)$.

\begin{prop}\label{prop;cong-ast-n-k}
Suppose $B\in V(A)$ is a set of rank $n+2$, definable from $A$ and parameters from $V$, and there is some injective map $\chi\colon B\to\mathcal{P}^{k+1}(\omega)$ such that any two distinct sets in the image of $\chi$ are disjoint.
Then $V(B)\subsetneq V(A)$.
\end{prop}

\begin{proof}
Assume otherwise.
As in Proposition~\ref{prop;cong31-not-30} we find a formula $\varphi$, condition $p\in P$, elements $U_1,...,U_m$ from $B$ and $w\in V(A^{n-1})$ such that
\begin{equation*}
    p\force \dot{g}\textrm{ is defined by } \varphi(\dot{g},\dot{A},\dot{U}_1,...,\dot{U}_m,\check{w}).
\end{equation*}
(We may use $\dot{A}$ instead of $\dot{B}$ by the definability assumption on $B$.)
Fix $a,b\in A^k$ not in the image of $p$.
Let $\pi^a_b$ be the permutation of $A^k$ swapping $a$ with $b$.
This generates a permutation of $P$ (sending $q\in P$ to $\pi^a_b\circ q$) which fixes $p$ and $\dot{A}$.
Applying the permutation to the statement above, it follows that $\pi^a_b\circ g$ is defined by $\varphi$ using the parameters $\pi^a_b U_1,...,\pi^a_b U_m\in B$.
Varying over $b$, we get an injective map between $A^k$ (minus $a$ and the image of $p$) into $B^m$.

Define $\chi'\colon B^m\to\mathcal{P}((\mathcal{P}^{k}(\omega))^m)$ by $\chi'(b_1,...,b_{m})=\set{(x_1,...,x_m)}{x_i\in \chi(b_i)}$.
Then any two distinct sets in the image of $\chi'$ are disjoint.
By composing $\chi'$ with a definable injective map from $(\mathcal{P}^k(\omega))^m$ into $\mathcal{P}^k(\omega)$, and with the map above, we get $\tilde{\chi}\colon A^k\to\mathcal{P}(\mathcal{P}^k(\omega))=\mathcal{P}^{k+1}(\omega)$ such that any two distinct sets in the image of $\tilde{\chi}$ are disjoint.

As in Observation~\ref{obs;KW-implies-low-generator}, we can code $A^k$ as a set of rank $k$ (one rank lower than $A^k$) using $\tilde{\chi}$ as follows. Let $C=\set{(x,y)}{\exists X\in A^k(y\in \tilde{\chi}(X)\wedge x\in X)}$, then $C$ is of rank $k$ and $A^k$ is definable from $C$, in particular, $A^k\in V(C)$.

It follows from Lemma~\ref{lem;stong-Monro-8} that $C$ is in $V(A^k)$ (the forcing to add $g$ to $V(A^k)$ adds no sets of rank $k$). 
From Lemma~\ref{lem;gen-blass-thm} it follows that $C\in V(A^{k-1})(\bar{d})$ where $\bar{d}$ is a finite subset of $A^{k}$. Thus $A^k\in V(A^{k-1})(\bar{d})$, a contradiction.
\end{proof}

\begin{cor}\label{cor;conj-part-1}
For $n\geq 2$, $1\leq k\leq n$, $\cong^\ast_{n+2,k}$ is not Borel reducible to $\cong^\ast_{n+2,k-1}$.
\end{cor}
\begin{proof}
As in the first two paragraphs of Section~\ref{subsec;cong-ast-n-k}, fix $n$ and $k$ and let $A$ be the $\cong^\ast_{n+2,k}$-invariant defined above (before Proposition~\ref{prop;cong-ast-n-k}).
By Lemma~\ref{lem;reduction-def-generating-set}, to show that $\cong^\ast_{n+2,k}$ is not Borel reducible to $\cong^\ast_{n+2,k-1}$, it suffices to show that $V(A)$ is not of the form $V(B)$ for a $\cong^\ast_{n+2,k-1}$-invariant $B\in V(A)$ which is definable from $A$ and parameters in $V$ alone.
By Claim~\ref{claim;cong-ast-invs}, it suffices to show that $V(A)$ is not of the form $V(B)$ when $B\in V(A)$ is a set of rank $n+2$, definable from $A$ and parameters in $V$ alone, such that there is some injective map from $B$ into $\mathcal{P}^{k+1}(\omega)$ such that any two disinct sets in the image are disjoint. The latter is precisely Proposition~\ref{prop;cong-ast-n-k} above.
\end{proof}

\section{Transfinite jumps}\label{sec;trans-jumps}
In this section we consider the Friedman-Stanley jumps above $\cong_\omega$, and the corresponding equivalence relations of Hjorth, Kechris and Louveau.
In order to prove Theorem~\ref{thm;HKL-conj}, following the ideas above, we first need to cast the irreducibility results along the transfinite Friedman-Stanley hierarchy in terms of symmetric models.
For example, in order to show that $\cong_{\omega+1}$ is not Borel reducible to $\cong_\omega$ we need to find a model generated by a set $A\in\mathcal{P}^{\omega+1}(\omega)$ yet not by any set in $\mathcal{P}^\omega(\omega)$.

The main difficulty is to continue Monro's construction past the $\omega$'th stage, as mentioned in Section~\ref{subsec;Kinna-Wagner}.
Suppose $\seqq{A^n}{n<\omega}$ are as above and $A^\omega=\bigcup_n A^n$.
Naively, if we were to force a subset of $A^\omega$ with finite conditions as before, this forcing will certainly add new sets of low rank.
E.g., a generic subset of $A_1$ will be added.
However, the fact that no sets of rank $\leq n$ were added to $V(A^n)$ was crucial in Monro's arguments.

To avoid adding a generic subset of any particular $A^n$ we will instead force to add a choice function $\seqq{a_n}{n<\omega}\in\prod_n A^n$.
If we do so naively, the real $\set{n}{a_n\in a_{n+1}}$ will be new. 
This will be the only difficulty: we show that adding $\seqq{a_n}{n<\omega}$ such that $a_n\in A^n$ and $a_n\in a_{n+1}$ works.

Adding infinitely many such sequences, to construct $A^{\omega+1}$, is significantly more complex, and a direct approach fails. 
For example, given $\seqq{a_n}{n<\omega}$ and $\seqq{b_n}{n<\omega}$, the reals $\set{n}{a_n=b_n}$ and $\set{n}{a_n\in b_n}$ may be new.
Similarly, we must prevent any non trivial interactions between any finitely many such sequences.

The solution will be to split the construction into two steps.
First we force to add a regular binary tree $T$ of approximations for generic choice functions in $\prod_n A^n$, which are sufficiently indiscernible.
This tree will have no branches.  
We then force an infinite set of branches through $T$, which will be $A^{\omega+1}$.

The proofs will rely on a fine analysis of the model $V(A^\omega)$.
We will define posets $P^{n,\omega}$ which will add the sequence $A^\omega$ over the model $V(A^n)$.
Unlike the step from $V(A^n)$ to $V(A^{n+1})$, these posets will add reals so we cannot argue as in Section~\ref{sec;Monro-models}.
Instead, we will argue that as $n$ increases the conditions of the posets $P^{n,\omega}$ are increasingly indiscernible and therefore do not do much damage in the limit.

Instead of iterating Monro's posets as in \cite{Kar16}, we will work in a single Cohen-real extension and code the sequence $\seqq{A^n}{n<\omega}$ there.
Another motivation for this approach is that by coding the invariants with a single Cohen real we find a topology to work with.
That is, the irreducibility results we are proving hold on any nonmeager set (see Section~\ref{sec;topologies}).

\subsection{The model $V(A^\omega)$ and a proof of $\cong^\ast_{\omega+1,0}<_B\cong^\ast_{\omega+1,<\omega}$}\label{subsec;cong-ast-omega+1}

An important feature of the atoms in Fraenkel's permutation models is that they are indiscernible over parameters from the ``pure'' part of the universe.
A similar intuition holds for the Cohen reals, though they are not actual indiscernibles.
One important feature of Monro's construction, although not explicit, is that the elements of $A^2$ {\it are} indiscernible over parameters from $V$.

We prove below a more general statement which will be used in the following section as well.
Roughly speaking, for elements $\bar{a}$ high in the hierarchy (from $\bigcup_{i>n}A^i$), for any statement about $\bar{a}$ involving low parameters (from $V(A^{n-1})$), the truth of this statement only depends on the $\in$-relations between the elements of $\bar{a}$.

\begin{defn}
Given a sequence $\bar{x}=x_1,...,x_n$ from $\bigcup_k A^k$, the {\bf type} of $\bar{x}$ is the structure $(n,E,\approx,\set{P_k}{k<\omega})$, defined by
\begin{itemize}
    \item $iEj \iff x_i\in x_j$;
    \item $i\approx j \iff x_i=x_j$;
    \item $i\in P_k \iff x_i\in A^k$.
\end{itemize}

That is, for two sequences $\bar{x}$ and $\bar{y}$, they have the {\bf same type} if and only if they have the same length and $x_i\in x_j\iff y_i\in y_j$, $x_i=x_j \iff y_i=y_j$, $x_i\in A^k\iff y_i\in A^k$, for $i,j=1,...,n$ and any $k<\omega$. 
We say that $\bar{x}$ and $\bar{y}$ have the {\bf same type over $\bar{a}$} if $\bar{a}^\frown \bar{x}$ and $\bar{a}^\frown \bar{y}$ have the same type.
\end{defn}

\begin{lem}[Indiscernibility in $V(A^m)$]\label{lem;indiscern}
Let $\psi$ be a formula, $n\leq m$, $v\in V(A^{n-1})$, $\bar{a}$ a finite subset of $A^n$ and $\bar{x}, \bar{y}$ finite subsets of $\bigcup_{k=n+1}^m A^k$.
Assume further that $\bar{x}$ and $\bar{y}$ have the same type over $\bar{a}$. 
Then 
\begin{equation*}
    V(A^m)\models \psi(A^m,v,\bar{a},\bar{x})\iff \psi(A^m,v,\bar{a},\bar{y}).
\end{equation*}
\end{lem}
\begin{proof}
Fix $v$ and $\bar{a}$.
The proof is by induction on $m$.
For the base case $m=n$ there is nothing to prove.
Assume the result for $m$ and work in $V(A^{m+1})$, where $\bar{x},\bar{y}$ are finite subsets of $\bigcup_{k=n+1}^{m+1}A^k$ with the same type.
Let $\bar{x}=\bar{x}_0,\bar{x}_m,\bar{x}_{m+1}$ where $\bar{x}_0\subset\bigcup_{k=n+1}^{m-1}A^k$, $\bar{x}_m\subset A^m$ and $\bar{x}_{m+1}\subset A^{m+1}$.
Similarly take $\bar{y}=\bar{y}_0,\bar{y}_m,\bar{y}_{m+1}$.
(In the case $m+1=n+1$, the only non-empty sequences are $\bar{x}_{m+1},\bar{y}_{m+1}$. In the arguments below, $\bar{a}$ will be used instead of both $\bar{x}_m$ and $\bar{y}_m$.)

By the assumption on types it follows that the lengths of $\bar{x}_m$ and $\bar{y}_m$ are the same. Similarly for $\bar{x}_{m+1},\bar{y}_{m+1}$.
We first argue that the lengths of $\bar{x}_{m+1}$ and $\bar{x}_m$ can be assumed to be the same (and therefore the same as $\bar{y}_m$ and $\bar{y}_{m+1}$).
If $\bar{x}_{m+1}$ is shorter than $\bar{x}_m$ add to $\bar{x}_{m+1}$ members of $A^{m+1}$ which include all of $\bar{x}_m$. Similarly, add to $\bar{y}_{m+1}$ members of $A^{m+1}$ which include all of $\bar{y}_m$.
The types are still the same, and dummy variables may be added to $\psi$.
If $\bar{x}_m$ is shorter than $\bar{x}_{m+1}$ add to $\bar{x}_m$ elements of $A^m$ which are inside each member of $\bar{x}_{m+1}$ and include all members of $\bar{x}_{0}$ from $A^{m-1}$.
Analogously expand $\bar{y}_{m}$.
The types remain the same and dummy variables may be added to $\psi$.

Assume now that $\psi(A^{m+1},v,\bar{a},\bar{x}_0,\bar{x}_m,\bar{x}_{m+1})$ holds in $V(A^{m+1})$. 
We need to show that $\psi(A^{m+1},v,\bar{a},\bar{y}_0,\bar{y}_m,\bar{y}_{m+1})$ holds as well.
Let $\bar{x}_{m+1}=\bar{x}_{m+1}(1),...,\bar{x}_{m+1}(l)$, $\bar{x}_m=\bar{x}_m(1),...,\bar{x}_m(l)$.
Recall that $V(A^{m+1})$ is an inner model of $V(A^m)[G]$, where $G=\seqq{A^{m+1}_c}{c\in A^m}$ is a $P(A^m)$-generic over $V(A^m)$.
By a finite permutation of $G$ (which preserves $A^{m+1}$), we may assume that $\bar{x}_{m+1}(j)=A^{m+1}_{\bar{x}_m(j)}$, for each $j$.
Let $p$ be a condition forcing that
\begin{equation*}
    \psi^{V(\dot{A}^{m+1})}(\dot{A}^{m+1},v,\bar{a},\bar{x}_0,\bar{x}_m,\dot{A}^{m+1}_{\bar{x}_m(1)},...,\dot{A}^{m+1}_{\bar{x}_m(l)}).
\end{equation*}
By Lemma~\ref{lem;strong-Monro-7}, the condition $r=p\rest\{\bar{x}_m(1),...,\bar{x}_m(l)\}^2$ already forces the statement.

Given $\bar{h}=h(1),...,h(l)$ from $A^{m}$, let $r[\bar{h}]$ be the condition $r$ with $\bar{x}_m(j)$ replaced by $h(j)$.
Let $\phi(A^m,v,\bar{a},\bar{x}_0,\bar{x}_m)$ be the statement that $r=r[\bar{x}_m]$ forces the displayed formula above.
By the inductive hypothesis, $\phi(A^m,v,\bar{a},\bar{y}_0,\bar{y}_m)$ holds in $V(A^m)$.
Thus
\begin{equation*}
    r[\bar{y}_m]\force \psi^{V(\dot{A}^{m+1})}(\dot{A}^{m+1},v,\bar{a},\bar{y}_0,\bar{y}_m,\dot{A}^{m+1}_{\bar{y}_m(1)},...,\dot{A}^{m+1}_{\bar{y}_m(l)}).
\end{equation*}
Let $G'$ be a finite permutation of $G$ so that in $V(A^m)[G'
]$, $A^{m+1}_{\bar{y}_m(j)}=\bar{y}_{m+1}(j)$ for each $j$.
The key point is the following: since $\bar{x}_m,\bar{x}_{m+1}$ and $\bar{y}_m,\bar{y}_{m+1}$ have the same type, the condition $r[\bar{y}_m]$ is in $G'$.
It follows that $\psi(A^{m+1},v,\bar{a},\bar{y}_0,\bar{y}_m,\bar{y}_{m+1})$ holds in $V(A^{m+1})$, as required.
\end{proof}

The indiscernibility lemma will be crucial below.
We now turn to the construction of the model $M_\omega=V(A^\omega)$ and establishing similar indiscernibility there.

\begin{defn}\label{def;monro-iterations}
Let $P^{n,\omega}$ be the poset of all finite partial functions $p\colon \omega\setminus n\times A^n\times A^n\lto \{0,1\}$, ordered by extension.
\end{defn}

A generic for $P^{n,\omega}$ over $V(A^n)$ produces a function $g\colon \omega\setminus n\times A^n\times A^n\lto \{0,1\}$.
Working in $V(A^n)[g]$, define $A^{n+1}_a=\set{b\in A^n}{g(n,a,b)=1}$ for $a\in A^n$ and $A^{n+1}=\set{A^{n+1_a}}{a\in A^n}$.
Inductively define $A^k$ for $k>n$ as follows.
Assume that $A^k$ has been defined and is indexed by $A^n$, $A^k=\set{A^k_a}{a\in A^n}$, define
\begin{equation*}
A^{k+1}_{a}=\set{A^k_{b}}{g(k,a,b)=1},\textrm{ and } A^{k+1}=\set{A^{k+1}_a}{a\in A^n}.
\end{equation*}
Recall that $P(A^n)$ is the poset to add a single function $A^n\times A^n\lto \{0,1\}$.
Define maps $f_n\colon P^{n,\omega}\lto P(A^n)\ast P^{n+1,\omega}$ as follows.
Let $p\in P^{n,\omega}$ be a condition.
Define $q\in P(A^n)$ by $q(a,b)=i\iff p(n,a,b)=i$.
Define a $P(A^n)$-name for a $P^{n+1,\omega}$ condition $r$ by $r(m,\dot{A}^{n+1}_a,\dot{A}^{n+1}_b)=i\iff p(m,a,b)=i$, for any $m\geq n+1$.
$f_n$ will map $p$ to $q\ast r$.

\begin{lem}\label{lem;monro-iterations-projections}
$f_n$ is a forcing isomorphism.
\end{lem}

We will work in a generic extension of $P^{0,\omega}$ over $V$. 
In this model we can construct $\seqq{A^n}{n<\omega}$ as above.
By applying the projections $f_n$ we get, for each $n$, a $P^{n,\omega}$-generic over $V(A^n)$ which produces the same sequence $\seqq{A^m}{m\geq n}$ (as described above).
In particular, for each $n$ there is an enumeration of $A^{n+1}$ indexed by $A^n$, $\seqq{A^{n+1}_a}{a\in A^n}$, which is $P(A^n)$-generic over $V(A^n)$.
It follows that this sequence satisfies the properties of Monro's construction.
Let $A^\omega=\bigcup_n A^n$.
We are interested in the model $M_\omega=V(A^\omega)$.

We want to prove, for instance, that rank $n$ sets in $V(A^\omega)$ are the same as in $V(A^n)$.
To that end, the poset $P^{n,\omega}$ will be used to present $V(A^\omega)$ in a generic extension of $V(A^n)$.
Unfortunately, the posets $P^{n,\omega}$ all add reals, thus are not well behaved for our purposes.
For example, given any $a\in A^n$, the set $\set{m}{g(m,a,a)=1}$ is generic.
The main point is that a parameter from $A^n$ was necessary to define this real.
The following claim shows that without using such high rank parameters, no new low rank sets can be defined.

\begin{claim}\label{claim;Monro-limit-const}
Let $\phi$ be a formula, $v\in V(A^{n-1})$.
Working in $V(A^{n+1})$, the value of $\phi(\dot{A}^\omega,v)$ is decided by the empty condition in $P^{n+1,\omega}$. 
\end{claim}
\begin{proof}
The point is that the conditions of the poset are indiscernible over the parameter $v$, thus cannot force conflicting statements.
The proof is analogous to Lemma~\ref{lem;strong-Monro-7}, using Lemma~\ref{lem;indiscern} instead of Lemma~\ref{lem;Monro-lem6}.
\end{proof}

\begin{cor}\label{cor;limit-monro-thm-8}
$V(A^\omega)\cap \mathcal{P}^n(\mathrm{On})=V(A^n)\cap \mathcal{P}^n(\mathrm{On})$.
\end{cor}
\begin{proof}
Any set $X$ in $V(A^\omega)$ is of the form $X=\set{x}{\psi(A^\omega,v,w,x)}$, where $w$ is finite subset of $A^\omega$ and $v\in V$.
The corollary is proved by induction on the rank $n$.
Assume $X$ is of rank $n+1$ and $X\subset V(A^n)$ by the inductive hypothesis.
Take $m$ large enough such that $w\subset V(A^{m-1})$ and $m>n$.
By the claim above, for any $x\in V(A^n)$, the statement $\psi^{V(\dot{A}^\omega)}(\dot{A}^\omega,v,w,\check{x})$ is decided by the empty condition in $P^{m+1,\omega}$.
It follows that $X$ is equal to 
\begin{equation*}
    \set{x\in V(A^n)}{P^{m+1,\omega}\force \psi^{V(\dot{A}^\omega)}(\dot{A}^\omega,v,w,\check{x}) },
\end{equation*}
which is in $V(A^{m+1})$.
Finally, it follows from Lemma~\ref{lem;Monro-thm8} that $X\in V(A^{n+1})$
\end{proof}

\begin{lem}[Indiscernibility in $V(A^\omega)$]\label{lem;indiscern-Aomega}
Let $\psi$ be a formula, $v\in V(A^{n-1})$, $\bar{a}$ a finite subset of $A^n$ and $\bar{x}, \bar{y}$ finite subsets of $\bigcup_{k=n+1}^m A^k$.
Assume further that $\bar{x}$ and $\bar{y}$ have the same type over $\bar{a}$. 
Then 
\begin{equation*}
    V(A^\omega)\models \psi(A^\omega,v,\bar{a},\bar{x})\iff \psi(A^\omega,v,\bar{a},\bar{y}).
\end{equation*}
\end{lem}
\begin{proof}
Work over $V(A^{m+2})$.
Let $\phi(A^{m+2},v,\bar{a},\bar{x})$ be the statement that $P^{m+2,\omega}$ forces $\psi^{V(\dot{A}^\omega)}(\dot{A}^\omega,v,\bar{a},\bar{x})$.
By Claim~\ref{claim;Monro-limit-const}, $\psi(A^\omega,v,\bar{a},\bar{x})$ holds in $V(A^\omega)$ iff $\phi(A^{m+2},v,\bar{a},\bar{x})$ holds in $V(A^{m+2})$.
Applying Lemma~\ref{lem;indiscern}, the latter is true if and only if $\phi(A^{m+2},v,\bar{a},\bar{y})$ holds in $V(A^{m+2})$.
This in turn holds if and only if $\psi(A^\omega,v,\bar{a},\bar{y})$ holds in $V(A^\omega)$.
\end{proof}

Before moving on to the $\omega+1$ step of the construction, we sketch a proof of $\cong^\ast_{\omega+1,n}<\cong^\ast_{\omega+1,<\omega}$, for every $n$.
Based on the proof in Section~\ref{subsec;cong-ast-31}, we want to force over $V(A^\omega)$ a generic $B\subset A^\omega$, without adding new sets of small rank.
The corresponding $\cong^\ast_{\omega+1,<\omega}$-invariant will be the set containing all finite changes of $B$.
The subset $B$ will be a choice function $\seqq{a_n}{n<\omega}$ added by the following poset.

Let $P$ be the poset of all finite functions $p\colon \dom p\lto \bigcup_n A^n$ where $\dom p\in \omega$, $p(i)\in A^i$  for every $i\in\dom p$ and $p(i)\in p(i+1)$ if $i+1\in \dom p$. 
$P$ is ordered by extension.

\begin{lem}
Let $\psi$ be a formula, $p\in P$, $v\in V(A^{m-1})$ such that $p\force \psi(A^\omega,v)$.
Then $p\rest (m+1)\force\psi(A^\omega,v)$.
\end{lem}
\begin{cor}
Forcing with $P$ does not add sets of rank $<\omega$ to $V(A^\omega)$.
\end{cor}
The point is that the conditions in $P$ have the same type and therefore are indiscernibles. The lemma and corollary are proved in more general setting below (Lemma~\ref{lem;homo-T} and Claim~\ref{claim;T-no-low-rank}).

Let $\seqq{a_n}{n<\omega}$ be a $P$-generic over $V(A^\omega)$.
Define
\begin{equation*}
    A=\set{\seqq{b_n}{n<\omega}\in\Pi_n A^n}{a_n=b_n\textrm{ for all but finitely many }n}.
\end{equation*}
$A$ is a set of rank $\omega+1$.
Furthermore, given any $Y\in A$, the map $Z\mapsto Z\Delta Y$ is injective and sends the members of $A$ to finite sequences of rank $<\omega$.
Thus $A$ is a $\cong^\ast_{\omega+1,<\omega}$-invariant.

We will work in the model $V(A^\omega)[\seqq{a_n}{n<\omega}]=V(\seqq{a_n}{n<\omega})=V(A)$. (Note that $A^\omega$ is definable from $\seqq{a_n}{n<\omega}$, since $A^n=\bigcup a_{n+2}$.)

\begin{lem}
Suppose $B\in V(A)$ is a set of rank $\omega+1$, definable from $A$ and parameters from $V$.
Assume further that for some $m$, there is an injective map between $X$ into $\mathcal{P}^m(\mathrm{On})$.
Then $V(X)\subsetneq V(A)$.
\end{lem}
The proof is similar to Proposition~\ref{prop;cong31-not-30}.
As in Section~\ref{subsec;cong-ast-31} we conclude:

\begin{cor}
For every $n\in\omega$,
    $\cong^\ast_{\omega+1,n}<\cong^\ast_{\omega+1,<\omega}$.
\end{cor}

\subsection{Proof of $\cong^\ast_{\omega+2,<\omega}<_B\cong^\ast_{\omega+2,\omega}$ and the model $V(A^{\omega+1})$}\label{subsec;cong-ast-omega+2}

To force good $\cong^\ast_{\omega+2,\beta}$-invariants for $\beta\leq\omega$, as in Section~\ref{sec;cong-ast-n-k}, we first need to have a good $\cong_{\omega+1}$-invariant, $A^{\omega+1}$, which will be a set of choice functions in $\prod_n A^n$ as added above.
We first add an auxiliary tree $T$ which will guide the forcing for adding such choice functions.


\begin{defn}
Define a poset $\mathcal{T}$ as follows.
Elements of $\mathcal{T}$ are finite sets $t\subset\bigcup_k A^k$ such that the graph $(t,\in)$ is a rooted tree, with root $t\cap A^0$, which is isomorphic to $2^{<n}$ for some $n\in\omega$.
Call this $n$ the {\bf height} of $t$.
For $t,u\in \mathcal{T}$, $t$ extends $u$ if $t\supset u$.
\end{defn}

If $t\in\mathcal{T}$, $n$ is the height of $t$ and $m\leq n$, define $t\rest{m}$ to be the subtree of $t$ of height $m$.
(If $m>n$, let $t\rest m=t$).
Say that $\bar{t}$ is an {\bf enumeration} of $t$ if $\bar{t}$ is a sequence of length $|t|$, enumerating $t$, such that lower rank sets appear before higher rank sets. 

\begin{lem}\label{lem;homo-T}
Let $\psi$ be a formula, $t\in\mathcal{T}$, $v\in V(A^{m-1})$ such that $t\force \psi(A^\omega,v)$.
Then $t\rest(m+1)\force \psi(A^\omega,v)$.
\end{lem}
\begin{proof}
We show that any $q\in\mathcal{T}$ extending $t\rest(m+1)$ can be extended to a condition forcing $\psi(A^\omega,v)$.
Let $t'$ be an enumeration of $t\rest (m+1)$, and $\bar{t}$ an enumeration of $t$ extending $t'$.
Let $h$ be the height of $t$.
Fix some $q$ extending $t\rest(m+1)$ and assume that its height is $\geq h$.
Let $\bar{q}$ be an enumeration of $q\rest h$ extending $t'$.

By the definition of the forcing $\mathcal{T}$, the sequences $\bar{t}$ and $\bar{q}$ have the same type.
Furthermore, they agree on all elements in $\bigcup_{j\leq m}A^j$, since these are in the initial segment corresponding to $t'$.
By indiscernibility, Lemma~\ref{lem;indiscern-Aomega}, for the statement $t\force \psi(A^\omega,v)$, it follows that $q\rest h\force \psi(A^\omega,v)$ and so $q\force \psi(A^\omega,v)$.
\end{proof}

\begin{claim}\label{claim;T-no-low-rank}
$\mathcal{T}$ adds no sets of rank $<\omega$ to $V(A^\omega)$.
\end{claim}
\begin{proof}
As usual the proof is by induction on rank.
Assume no rank $n$ sets are added, and $B$ is of rank $n+1$.
Take $m$ large enough so that the parameters defining $\dot{B}$ in $V(A^\omega)$ are in $V(A^{m-1})$.
Let $T$ be some $\mathcal{T}$-generic and $t\in\mathcal{T}$ be its subtree of height $m+1$.
Then $B$ can be defined in $V(A^\omega)$ as $B=\set{x}{t\force \check{x}\in\dot{B}}$ 
(as in the proofs in Lemma~\ref{lem;stong-Monro-8} and Corollary~\ref{cor;limit-monro-thm-8}).
\end{proof}

Let $T$ be $\mathcal{T}$-generic over $V(A^\omega)$, we now work in the model $V(A^\omega)[T]=V(T)$.
The following lemma is the heart of the matter, showing that the nodes of $T$ are sufficiently indiscernible in $V(T)$.

\begin{lem}[Indiscernibility in $V(T)$]\label{lem;cones-indisc}
Let $\psi$ be a formula, $v\in V(A^{n-1})$ and $\bar{u}=u_1,...,u_k$ distinct elements in $T$ of level $n$.
Suppose $u^i_j$ is in level $l$ of the tree and above $u_j$, for $i=0,1$ and $j=1,...,k$.
Then
\begin{equation*}
    V(T)\models \psi(T,v,\bar{u}^0)\iff \psi(T,v,\bar{u}^1).
\end{equation*}
\end{lem}
\begin{proof}
Assume that $\psi(T,v,\bar{u}^0)$ holds, and pick some $t\in\mathcal{T}$ of height $\geq l$, compatible with $T$, forcing this.
We show that $t\force\psi(\dot{T},v,\bar{u}^1)$.
Fix some enumeration $v'$ of the levels of $t$ below $n$, $\bar{a}$ of the n'th level of $t$, and $\bar{x}$ of the higher levels.
Working in $V(A^\omega)$, let $\phi(A^\omega,v,v',\bar{a},\bar{x})$ be the formula saying that the condition $t$ corresponding to $v',\bar{a},\bar{x}$ forces that $\psi(\dot{T},v,\bar{u}^0)$ holds, where $\bar{u}^0$ is identified as a subsequence of $\bar{x}$.

We will use the following simple fact: given a regular rooted binary tree and two nodes $p,q$ on the same level, there is an automorphism of the tree sending $p$ to $q$.
For any $j=1,...,k$, consider the tree $t$ restricted to the cone above $u_j$ as a tree with root $u_j$.
By applying the fact just mentioned, and combining these automorphisms, we get an automorphism of $t$ sending $u_j^0$ to $u_j^1$, and preserving the levels $\leq n$.
Applying this automorphism to the enumeration $v',\bar{a},\bar{x}$ of $t$, we get a different enumeration $v',\bar{a},\bar{y}$.

That we have used an automorphism of the tree $(t,\in)$ precisely ensures that $\bar{a},\bar{x}$ and $\bar{a},\bar{y}$ have the same type.
Applying Lemma~\ref{lem;indiscern-Aomega}, it follows that $\phi(A^\omega,v,v',\bar{a},\bar{y})$ holds in $V(A^\omega)$.
The indices in $\bar{x}$ which correspond to $\bar{u}^0$ correspond to $\bar{u}^1$ in $\bar{y}$.
This means that the corresponding tree to $v',\bar{a},\bar{y}$, which is $t$, forces that $\psi(\dot{T},v,\bar{u}^1)$ holds.
\end{proof}

\begin{defn}
Let $\mathcal{B}$ be the poset of all finite functions $p\colon \dom p\lto T$ with $\dom p\subset\omega$.
For $p,q\in\mathcal{B}$, $p$ extends $q$ if $\dom p\supset \dom q$ and $p(k)$ is above $q(k)$ in $T$ for any $k\in\dom q$.
For every $k\in\omega$, let $\mathcal{B}_k$ be the poset of all $p\in\mathcal{B}$ with $\dom p=\{0,...,k-1\}$.
\end{defn}
If $p\in\mathcal{B}_k$ is such that $p(0),...,p(k-1)$ are all above level $n$, define $p\rest n$ to be the condition $p(0)\rest n,...,p(k-1)\rest n$.
Otherwise $p\rest n=p$.
\begin{lem}\label{lem;homo-B-k}
Suppose $p\in\mathcal{B}_k$, $n<\omega$ such that the projections of $p(0),...,p(k-1)$ to level $n$ of the tree are distinct.
Let $v\in V(A^{n-1})$ and $\psi$ a formula such that $p\force \psi(T,v)$.
Then $p\rest n\force \psi(T,v)$.
\end{lem}
\begin{proof}
We may assume that $p(0),...,p(k-1)$ are all in the same level $l$ of the tree.
Let $u_j$ be the restrictions of $p(j)$ to level $n$ of the tree, for $j<k$.
We show that any condition $q$ extending $p\rest n$ such that $q(0),...,q(k-1)$ are in level $l$ of the tree forces $\psi(T,v)$.
This will finish the proof since such conditions and pre-dense below $p\rest n$.
Fix such $q\in\mathcal{B}_k$.
Let $u^0_j=p(j)$, $u^1_j=q(j)$ for $j<k$. 
By Lemma~\ref{lem;cones-indisc} for the statement $p\force \psi(T,v)$, it follows that $q\force \psi (T,v)$.
\end{proof}

\begin{cor}\label{cor;B-k-no-low-rank}
For any $k$, forcing with $\mathcal{B}_k$ adds no sets of rank $<\omega$ to $V(T)$.
\end{cor}
\begin{proof}
The proof is similar to that of Claim~\ref{claim;T-no-low-rank}, using Lemma~\ref{lem;homo-B-k} instead of Lemma~\ref{lem;homo-T}.
\end{proof}

Let $B$ be $\mathcal{B}$-generic over $V(T)$ and $A^{\omega+1}=\set{B(n)}{n\in\omega}$. Define
\begin{equation*}
    M_{\omega+1}=V(T)(A^{\omega+1})=V(A^{\omega+1}).
\end{equation*}

\begin{claim}
Any sequence $a_0,...,a_{k-1}$ of distinct members from $A^{\omega+1}$ is $\mathcal{B}_k$-generic over $V(T)$.
\end{claim}
\begin{proof}
Working in $V(T)[B]$, fix $i_0,...,i_{k-1}$ such that $B(i_j)=a_j$.
The claim follows since the map from $\mathcal{B}$ to $\mathcal{B}_k$ sending $p$ to $p(i_0),...,p(i_{k-1})$ is a complete projection (on a dense open set).
\end{proof}

\begin{prop}
$V(A^{\omega+1})$ and $V(A^\omega)$ have the same sets of rank $<\omega$.
\end{prop}
\begin{proof}
By Lemma~\ref{lem;symmetry} any set $X\in V(A^{\omega+1})$ which is a subset of $V(T)$ is in $V(T)[a_1,...,a_k]$ for finitely many $a_1,...,a_k\in A^{\omega+1}$.
By the claim above, $V(T)[a_1,...,a_k]$ is a $\mathcal{B}_k$-generic extension of $V(T)$, thus agrees with $V(T)$ on rank $<\omega$ elements, by Corollary~\ref{cor;B-k-no-low-rank}.
An inductive argument as in Lemma~\ref{lem;stong-Monro-8} shows that $V(A^{\omega+1})$ and $V(T)$ have the same sets of rank $<\omega$. 
This finishes the proof by Claim~\ref{claim;T-no-low-rank}.
\end{proof}

\begin{cor}
Suppose $X\in V(A^{\omega+1})$ is of rank $\omega$. 
Then there are finitely many $a_1,...,a_k\in A^{\omega+1}$ such that $X\in V(T)[a_1,...,a_k]$.
In particular, $V(X)\subsetneq V(A^{\omega+1})$.
\end{cor}
\begin{proof}
By the proposition above a set of rank $\omega$ is contained in $V(T)$, so the result follows from Lemma~\ref{lem;symmetry}.
\end{proof}

\begin{cor}
$\mathrm{KWP}^{\omega}$ fails in $V(A^{\omega+1})$.
\end{cor}
\begin{proof}
The proof follows from Observation~\ref{obs;KW-implies-low-generator} and the previous corollary.
\end{proof}
The following lemma is the $\omega+1$ stage analogue of Lemma~\ref{lem;Monro-lem6}, which was the heart of Monro's arguments.
That is, it shows that the construction can be now carried through stages $\omega+2,\omega+3,...$
\begin{lem}\label{lem;Monro-lem6-Aomega+1}
Let $\psi$ be some formula, $x\in V(A^{n-1})$, $n<\omega$.
Assume $\{s_1,...,s_m\}$ are pairwise distinct members of $A^{\omega+1}$ and $\psi(A^{\omega+1},x,s_1,...,s_m)$ hold in $V(A^{\omega+1})$.
Then there are pairwise distinct $u_1,...,u_m$ in $T$ such that for any $t_1,...,t_m$ from $A^{\omega+1}$, if $u_i\in t_i$ then $\psi(A^{\omega+1},x,t_1,...,t_m)$ holds.
\end{lem}
\begin{proof}
Working in $V(T)[B]$, 
fix $i_1,...,i_m$ such that $s_j=B(i_j)$.
Let $p\in B$ such that
\begin{equation*}
 p\force \psi^{V(\dot{A}^{\omega+1})}(\dot{A}^{\omega+1},\check{x},\dot{B}(i_1),...,\dot{B}(i_m)).   
\end{equation*}
Let $u_j=p(i_j)$ for $j=1,...,m$, which we may assume are in the same level of $T$.
Assume towards a contradiction that $t_1,...,t_m$ are in $A^{\omega+1}$, $u_i\in t_i$, but $\psi(A^{\omega+1},x,t_1,...,t_m)$ fails.
Fix $e_1,...,e_m$ such that $t_j=B(e_j)$ and a condition $q\in B$ such that
\begin{equation*}
    q\force\neg \psi^{V(\dot{A}^{\omega+1})}(\dot{A}^{\omega+1},\check{x},\dot{B}(e_1),...,\dot{B}(e_m))
\end{equation*}
For notational simplicity, assume that $\{i_1,...,i_m\}$ and $\{e_1,...,e_m\}$ are disjoint.
Let $\pi$ be a finite support permutation of $\omega$ sending $e_j$ to $i_j$ and $i_j$ outside the domain of $p$, for each $j=1,...,m$.
$\pi$ generates a permutation of $\mathcal{B}$ which fixes $\dot{A}^{\omega+1}$.
Applying $\pi$ to the statement above,
\begin{equation*}
 \pi q\force \neg\psi^{V(\dot{A}^{\omega+1})}(\dot{A}^{\omega+1},\check{x},\dot{B}(i_1),...,\dot{B}(i_m)).   
\end{equation*}
However, $\pi q$ is compatible with $p$, a contradiction.
\end{proof}
Let $P$ be the poset of finite partial functions from $A^{\omega+1}$ to $\{0,1\}$, to add a generic subset of $A^{\omega+1}$.

\begin{claim}
$P$ adds no sets of rank $\omega$ to $V(A^{\omega+1})$.
\end{claim}
\begin{proof}
The proof follows the same arguments as in Lemma~\ref{lem;stong-Monro-8} (see the proofs in section~\ref{subsec;Monro-posets-tech}), using Lemma~\ref{lem;Monro-lem6-Aomega+1} above instead of Lemma~\ref{lem;Monro-lem6}.
\end{proof}

Let $X$ be $P$-generic over $V(A^{\omega+1})$.
Define 
\begin{equation*}
    A=\set{X\Delta \bar{a}}{\bar{a}\subset A^{\omega+1}\textrm{ is finite}}.
\end{equation*}
Given any $Y\in A$, the map $Z\mapsto Z\Delta Y$ is injective, sending members of $A$ to finite sets whose elements are of rank $\omega$.
Thus $A$ is a $\cong^\ast_{\omega+2,\omega}$-invariant.

\begin{lem}\label{lem;omega+2,fin-not-omega+2,omega}
Suppose $B\in V(A)$ is a set of rank $\omega+2$, definable from $A$ and parameters from $V$.
Assume further that there is an injective map from $B$ into $\mathcal{P}^{<\omega}(\mathrm{On})$.
Then $V(B)\subsetneq V(A)$.
\end{lem}
\begin{proof}
The proof follows the same outline as Proposition~\ref{prop;cong31-not-30}.
\end{proof}

\begin{cor}
    $\cong^\ast_{\omega+2,<\omega}<_B\cong^\ast_{\omega+2,\omega}$.
\end{cor}

\section{The general case}\label{sec;general-conj-proof}
In this section we continue Monro's construction through all the countable ordinals, thus proving Theorem~\ref{thm;Kinna-Wagner-separation}.
Combined with the techniques from Section~\ref{sec;cong-ast-n-k} we then establish parts (2) and (3) of Conjecture~\ref{conj;HKL}.
This section contains sketches of the arguments.
The focus will be on the few basic ideas that require adaptation.
Based on these changes the details are similar to those presented above.

First we mention a difficulty in generalizing the arguments of Section~\ref{subsec;cong-ast-omega+2} for higher countable ordinals.
Recall that in section~\ref{subsec;cong-ast-omega+1}, in order to add one subset of $A^\omega$, we forced a choice sequence $\seqq{a_n}{n<\omega}\in\prod_n A^n$ by finite approximations.
To avoid adding new sets of small rank, the conditions of the poset need to be sufficiently indiscernible.
To that end, we restricted to those sequences which are $\in$-increasing.
Another solution is to force with finite sequences from the even indices, to add a choice function in $\prod_n A^{2n}$.
By Lemma~\ref{lem;indiscern-Aomega}, the conditions are sufficiently indiscernible.

However, when trying to add many subsets of $A^\omega$ the indiscernibility of the higher levels of $A^\omega$ relative to the lower ones leads to adding new reals (as argued in the beginning of Section~\ref{sec;trans-jumps}).
The tree $T$ was added precisely to \textit{restrict} the indiscernibility by creating relations between elements in higher levels and lower levels.
These relations were based on the $\in$ relation between consecutive levels.

For stage $\omega+\omega$, the construction above would generalize without difficulty.
For limits of limit ordinals, such as $\omega\cdot\omega$, we want to fix a cofinal sequence $\alpha_n<\omega\cdot\omega$, and construct a tree $T$ as in Section~\ref{subsec;cong-ast-omega+2}, with level $n$ in $A^{\alpha_n}$.
For infinitely many $n$, $\alpha_{n+1}$ jumps above $\alpha_n+1$, so the $\in$ relation cannot be used.

The solution is to add a generic tree relation along with the finite approximations to the tree.
First we demonstrate how such construction would work at $\omega$.
Fix an increasing sequence $\seqq{\alpha_n}{n<\omega}$, cofinal in $\omega$, such that $\alpha_{n+1}>\alpha_n+1$ for all $n$.
Consider the poset $\mathcal{T}$ of pairs $(t,R)$ such that $t$ is a finite subset of $\bigcup_n A^{\alpha_n}$, and $R$ is a relation on $t$ such that for some natural number $m$, $(t,R)$ is isomorphic as a rooted tree to $(2^{<m},\sqsubset)$, with root $t\cap A^{\alpha_0}$.
A condition $(t,R)$ extends $(s,Q)$ if $s\subset t$ and $R\rest s\times s=Q$.

The indiscernibility lemma~\ref{lem;indiscern-Aomega} assures that the conditions in $\mathcal{T}$ are sufficiently indiscernible.
As in Section~\ref{subsec;cong-ast-omega+2}, it follows that adding the tree does not add small rank sets.
Furthermore, one can prove analogous indiscernibility lemmas for the model with the tree, e.g. Lemma~\ref{lem;cones-indisc}.
Using such lemma it follows that an infinite set of branches can be added as in Section~\ref{subsec;cong-ast-omega+2}.

\subsection{Invariants for $\cong_\lambda$}\label{subsec;cong-lambda}

Let $\lambda$ be a countable ordinal.
We will add trees through the limit ordinals $\delta<\lambda$ which are completely disjoint from one another.
Fix a sequence $\seqq{C_\delta}{\delta<\lambda\textrm{ is a limit ordinal}}$ where each $C_\delta\colon\omega\lto\delta$ is cofinal in $\delta$.
That is, a partial \textbf{ladder system}.
We require further that
\begin{equation*}
    \textrm{for any }\delta,\gamma,n,m,\textrm{ if }(\delta,n)\not=(\gamma,m)\textrm{ then }C_\delta(n)\notin\{C_\gamma(m)-1,C_\gamma(m),C_\gamma(m)+1\}.
\end{equation*}
This ensures that the cofinal sequences are disjoint and sufficiently indiscernible.
Such condition is not possible for a ladder system on $\omega_1$, and so our construction does not produce a stage $\omega_1$ model.

We carry a construction of $\seqq{A^\alpha}{\alpha<\lambda}$.
At limit stages $\delta<\lambda$ we add a tree $T_\delta$ whose levels are in $\seqq{A^{C_\delta(n)}}{n<\omega}$, as described above.
Similar indiscernibility lemmas as in Section~\ref{sec;trans-jumps} can be established.
Each tree only affects the indiscernibility for elements in the levels of the tree, or adjacent levels.
The condition on the sequences $C_\delta$ above ensures that at stage $\delta$ the levels in the tree $T_\delta$ still satisfy indiscernibility in $V(A^\delta)$ (where $A^\delta=\bigcup_{n<\omega}A^{C_\delta(n)}$), so adding the tree $T_\delta$ does not add low rank sets.

We then add an infinite set of branches $A^{\delta+1}$ through $T_\delta$.
An analogue of Lemma~\ref{lem;Monro-lem6-Aomega+1} can be verified, which allows to continue and define $A^{\delta+2}, A^{\delta+3},...$ as in Section~\ref{sec;Monro-models}.

\subsection{Invariants for $\cong^\ast_{\alpha+2,\beta}$}\label{subsec;cong-ast-lambda+2-alpha}

For any countable ordinal $\alpha$ and $\beta\leq\alpha$, we will construct a good $\cong^\ast_{\alpha+2,\beta}$-invariant as in Section~\ref{subsec;cong-ast-n-k}.
The crucial point is that we can force a function $A^{\alpha+1}\lto A^\beta$ without adding new sets of rank $\alpha$.
E.g., for the case $\alpha=\omega$ the arguments of Section ~\ref{subsec;Monro-posets-tech} can be repeated based on Lemma~\ref{lem;Monro-lem6-Aomega+1}.
Similar indiscernibility lemmas can be established for higher $\alpha$.
Working with a specific $\beta$, it will be convenient to fix a ladder system as above such that neither $\beta-1$, $\beta$ or $\beta+1$ appear as $C_\delta(n)$ for any $\delta$ and $n$.
This verifies part (3) of Conjecture~\ref{conj;HKL}.

\subsection{Invariants for $\cong^\ast_{\lambda+1,\beta}$}\label{subsec;cong-ast-omega+1-k}
It remains to construct good invariants for $\cong^\ast_{\lambda+1,\beta}$ for a limit ordinal $\lambda$ and $\beta<\lambda$.
In this case we cannot simply add a function $A^\lambda\to A^\beta$ without adding low rank sets. Recall that we want to add such a function to have an interesting action of $[A^\beta]^{<\aleph_0}$ on a set of rank $\lambda+1$ (see Section~\ref{subsec;cong-ast-n-k}). 
Instead, we will make modifications to the construction described in Section~\ref{subsec;cong-lambda} to produce such action.
The ideas are all present in the case $\lambda=\omega$ which we describe first.

Fix $\beta=k<\omega$.
As mentioned above we cannot add a function $A^\omega\lto A^k$ without adding low rank sets.
To overcome this problem we will make adjustments along the Monro construction.
At stages $n>k+2$, a generic function $g_n\colon A^n\lto A^k$ will be added, to provide an action of $[A^k]^{<\aleph_0}$ as in Section~\ref{subsec;cong-ast-n-k}.
To make the proofs easier, we will only add such function at odd stages, thus having more indiscernibility.

For example, at stage $k+3$, working in $V(A^{k+3})$, add a generic function $g_{k+3}\colon A^{k+3}\lto A^k$.
Let $\Pi$ be all finite permutations of $A^k$, and $\hat{A}^{k+3}=\set{\pi\circ g_{k+3}}{\pi\in \Pi}$.
We then continue the construction over the model
\begin{equation*}
    V(A^{k+3})[g_{k+3}]=V(A^{k+3})(\hat{A}^{k+3})=V(\hat{A}^{k+3}),
\end{equation*}
adding the set $A^{k+4}$ the same way as in Monro's construction.
That is, a set of generic subsets of $A^{k+3}$.

The main point is showing that the conditions of the poset $P(A^{k+3})$, for adding subsets of $A^{k+3}$, are sufficiently indiscernible in the model $V(\hat{A}^{k+3})$.
For example, Lemma~\ref{lem;Monro-lem6} will hold under the additional assumption that $g_{k+3}(t_i)=g_{k+3}(s_i)$.
Since for any $s_i$ there are infinitely many such $t_i$, Lemma~\ref{lem;strong-Monro-7} still holds, even when $g_{k+3}$ is used as a parameter.
Note that we use here Lemma~\ref{lem;indiscern}, that the elements in the domain of $g_{k+3}$ are indiscernible over the range, which is $A^k$.

Similarly we add $A^{k+5}$, a generic set of subsets of $A^{k+4}$ over $V(\hat{A}^{k+3},A^{k+4})$. 
Next we add a function $g_{k+5}\colon A^{k+5}\lto A^{k}$ generic over $V(\hat{A}^{k+3},A^{k+5})$, and continue in a similar fashion.
Note that the elements of $A^{k+5}$ are indiscernible over $V(\hat{A}^{k+3})$. 
Since no function was added at stage $n+4$, stage $n+5$ is very similar to stage $n+3$ described above.
Inductively, one can prove analogues of lemmas \ref{lem;Monro-lem6} and \ref{lem;strong-Monro-7}, as well as of the indisicernibility lemma \ref{lem;indiscern}.

Finally, we jump to the limit model as in Section~\ref{subsec;cong-ast-omega+1}.
Define $A^\omega= \bigcup_n A^n$ and $\hat{A}^\omega=\bigcup\{\hat{A}^n;\,n=k+3+i,\,i\textrm{ is even}\}$, and work in the model $V(A^\omega,\hat{A}^\omega)$.
In this model, for arbitrary large $n$, we have an ``interesting'' set of rank $n$, $\hat{A}^n$, and an interesting action of $A^k$ on this set.
Let $P$ be the poset of finite functions $p$ such that $p(n)\in \hat{A}^n$.
$P$ does not add $<\omega$-rank sets by similar arguments as in Section~\ref{subsec;cong-ast-omega+1}.
Let $g$ be a $P$-generic, then $\dom g=\set{k+3+i}{i\textrm{ is even}}$ and $g(n)\in \hat{A}^n$.
Let $\hat{\Pi}$ be all finite sequences from $\Pi$.
$\pi=\seqq{\pi_i}{i<m}$ in $\hat{\Pi}$ acts on $g$ by swapping $g(k+3+i)$ with $\pi_i\cdot g(k+3+i)$.
Define
\begin{equation*}
    A=\set{\pi\cdot g}{\pi\in\hat{\Pi}}.
\end{equation*}
We claim that $A$ is a $\cong^\ast_{\omega+1,k}$-invariant and that $V(A^\omega,\hat{A}^\omega)(A)=V(A)$ is not of the form $V(B)$ for any $\cong^\ast_{\omega+1,k-1}$-invariant. 
The argument is similar to Proposition~\ref{prop;cong-ast-n-k}.

For example, one important property that was used in Proposition~\ref{prop;cong-ast-n-k} is that for any condition $p$ there is a permutation $\pi$ fixing $p$ and $A$ yet changing $g$.
The same fact is crucial here, and is true for a different reason.
Given a condition $p$, it makes finitely many choices of elements in $\hat{A}^n$.
Take $\pi=\seqq{\pi_i}{i<m}\in\hat{\Pi}$ such that $\pi_i$ is the identity if $k+3+i\in\dom p$, and $\pi_i$ is not the identity for some $i$.
Then $\pi$ is as desired.

The discussion above verifies part (2) of Conjecture~\ref{conj;HKL} for $\lambda=\omega$.
The proof for arbitrary limit $\lambda<\omega_1$ can be done by combining the construction in this section and the construction of $\cong_\lambda$-invariants, as follows.

Fix $\beta<\lambda$.
Fix a ladder system $\seqq{C_\delta}{\delta\leq\lambda}$ as in Section~\ref{subsec;cong-lambda} such that $C_\lambda(0)>\beta$.
We use this ladder system to construct $\seqq{A^\alpha}{\alpha<\lambda}$ as in Section~\ref{subsec;cong-lambda}. 
Add now generic functions $g_n\colon A^{C_\lambda(n)}\to A^\beta$ as above, and let $\hat{A}^{C_\lambda(n)}=\set{\pi\circ g_n}{\pi\in\Pi}$ where $\Pi$ is all finite permutations of $A^\beta$.
In this case there is no need to skip the even stages: by the requirements on the ladder system the values of $C_\lambda(n)$ are sufficiently far apart and therefore the elements of $A^{C_\lambda(n+1)}$ are sufficiently indiscernible over the elements in $A^{C_\lambda(n)}$ and $\hat{A}^{C_\lambda(n)}$.
It follows that we can add the function $g_{n+1}$ without adding smaller rank sets (at this stage).
As described before (and as done to construct a $\cong^\ast_{\omega+1,<\omega}$-invariant in Section~\ref{subsec;cong-ast-omega+1}), we add a generic $g\in\prod_{n<\omega}\hat{A}^{C_\lambda(n)}$ and consider $A=\set{\pi\cdot g}{\pi\in\hat{\Pi}}$, where $\hat{\Pi}$ is the set of all finite sequences from $\Pi$. 
Then $A$ is our good invariant for $\cong^\ast_{\lambda+1,\beta}$.


\begin{thebibliography}{1}
    \bibitem[BK96]{BK96} H. Becker, A.S. Kechris, The Descriptive Set Theory of Polish Group Actions, Cambridge University Press, Cambridge, 1996.


    \bibitem[Bla81]{Bla81} Blass, Andreas. The Model of Set Theory Generated by Countably Many Generic Reals. The Journal of Symbolic Logic, vol. 46, no. 4, 1981, pp. 732-752.
    
    \bibitem[Coh63]{Coh63} Paul J. Cohen, The independence of the continuum hypothesis. I, Proceedings of the National Academy of Sciences, USA 50 (1963) 1143-1148. 
    
    
    \bibitem[Fra22]{Fra22} Fraenkel, A. (1922), Uber den Begriff ‘definit’ und die Unabhangigkeit des Auswahlsaxioms, Sitzungsberichte der Königlich Preussischen Akademie der Wissenschaften: 253-257.
    
    \bibitem[Fra37]{Fra37} A. A. Fraenkel, Ueber eine obgeschwachte Fassung des Auswahlaxioms, Journ. Symb. Logic 2 (1937), pp. 1-25.
    
    \bibitem[FS89]{FS89} Harvey Friedman and Lee Stanley. A Borel reducibility theory for classes of countable structures. J. Symbolic Logic, 54(3):894-914, 1989.
    
    \bibitem[Gao09]{Gao09} S. Gao, Invariant descriptive set theory, CRC Press, 2009.

    
    \bibitem[Gri75]{Gri75} Serge Grigorieff, Intermediate submodels and generic extensions in set theory, Ann. Math. (2) 101 (1975), 447-490.
    
    \bibitem[HL64]{HL64} Halpern, J. D.; Levy, A. The Boolean prime ideal theorem does not imply the axiom of choice. 1971 Axiomatic Set Theory (Proc. Sympos. Pure Math., Vol. XIII, Part I, Univ. California, Los Angeles, Calif., 1967) pp. 83-134 Amer. Math. Soc., Providence, R.I.
    
    \bibitem[Hjo00]{Hjo00} Greg Hjorth, Classification and orbit equivalence relations. Mathematical Surveys and Monographs, vol. 75. American Mathematical Society, Providence, RI, 2000.
    
    \bibitem[HKL90]{HKL90} L.A. Harrington, A.S. Kechris, and A. Louveau, A Glimm-Effros dichotomy for Borel equivalence relations, J. Amer. Math. Soc., 3(4) (1990), 903-928.
    
    
    
    
    \bibitem[HK96]{HK96} G. Hjorth and A.S. Kechris, Borel equivalence relations and classifications of countable models, Ann. Pure Appl. Logic, 82 (1996), 221-272.
    
    \bibitem[HK97]{HK97} Hjorth, Greg; Kechris, Alexander S. New Dichotomies for Borel Equivalence Relations. Bull. Symbolic Logic 3 (1997), no. 3, 329--346.
    
    \bibitem[HKL98]{HKL98} G. Hjorth, A. S. Kechris and A. Louveau, Borel equivalence relations induced by actions of the symmetric group, Ann. Pure Appl. Logic 92 (1998), 63-112.
    
    \bibitem[HR98]{HR98} P. Howard and J.E. Rubin, Consequences of the axiom of choice, Mathematical Surveys and Monographs, no. v. 1; v. 59, American Mathematical Society, 1998.

    \bibitem[Jec73]{Jec73} Thomas J. Jech, The axiom of choice, North-Holland Publishing Co., Amsterdam, 1973, Studies in Logic and the Foundations of Mathematics, Vol. 75.
    
    \bibitem[Jec03]{Jec03} Thomas Jech, Set theory. The third millennium edition, revised and expanded, Springer Monographs in Mathematics, Springer-Verlag, Berlin, 2003.
    
    \bibitem[Kan06]{Kan06} Kanamori, Akihiro. Levy and set theory., Ann. Pure Appl. Logic 140 (2006), 233-252.
    
    \bibitem[Kan08]{Kan08} Kanamori, Akihiro. Cohen and set theory. Bull. Symbolic Logic 14 (2008), no. 3, 351-378.
    
    \bibitem[Kano08]{Kano08} V. Kanovei, Borel equivalence relations, Amer. Math. Soc., 2008.
    
    \bibitem[KSZ13]{ksz} Kanovei V., Sabok M., Zapletal J.: Canonical Ramsey theory on Polish Spaces. Cambridge University Press, Cambridge (2013).
    
    \bibitem[Kar19]{Kar16} Asaf Karagila, Iterating Symmetric Extensions. J. Symb. Log. 84 (2019) no. 1, pp. 123-159.
    
    \bibitem[Kar18]{Kar17} Asaf Karagila, The Bristol Model: an abyss called a Cohen real. J. Math. Log. 18 (2018) no. 2, 1850008, 37 pp.
    
    \bibitem[Kec92]{Kec92} Alexander S. Kechris, The structure of Borel equivalence relations in Polish spaces, Set theory of the continuum (Berkeley, CA, 1989)	Math. Sci. Res. Inst. Publ., vol. 26, Springer, New York, 1992, pp. 89-102.
    
    
    \bibitem[KL97]{KL97} A. S. Kechris and A. Louveau, The classification of hypersmooth Borel equivalence relations, Journal of the American Mathematical Society, vol. 10 (1997), no. 1, pp. 215-242.
    
    
    \bibitem[KW55]{KW55} W. Kinna and K. Wagner, Über eine Abschwächung des Auswahlpostulates, Fund. Math. 42 (1955), 75-82.
    

    \bibitem[LZ20]{zprep} Paul B. Larson and Jindrich Zapletal, Geometric set theory, Mathematical Surveys and Monographs, vol. 248, American Mathematical Society, Providence, RI, 2020   
    
    \bibitem[Mon73]{Mon73} G. P. Monro, Models of ZF with the same sets of sets of ordinals, Fund. Math. 80 (1973), no. 2, 105-110.
    
    \bibitem[Mos39]{Mos39} Mostowski, Andrzej. "Über die Unabhängigkeit des Wohlordnungssatzes vom Ordnungsprinzip." Fundamenta Mathematicae 32.1 (1939): 201-252. 
    
    
    \bibitem[Sha$\infty$]{Sha18} Assaf Shani, Strong ergodicity around countable products of countable equivalence relations, arXiv e-prints 1910.08188.
    
    \bibitem[Sha19]{Sha19} Assaf Shani, PhD thesis, UCLA 2019.    
    
    
    \bibitem[WDR16]{WDR16} H. Woodin, J. Davis, and D. Rodriguez, The HOD Dichotomy, arXiv e-prints 1605.00613 (2016), 1-19.
    
    \bibitem[Zap01]{Zap01} Zapletal J., “Terminal notions in set theory,” Ann. Pure Appl. Logic, vol. 109, 89-116 (2001).
    
    \bibitem[Zap08]{Zap08} Jindrich Zapletal. Forcing Idealized. Cambridge Tracts in Mathematics 174. Cambridge University Press, Cambridge, 2008.

\end{thebibliography}
\end{document}